\documentclass[preprint,12pt]{elsarticle}

\usepackage{graphicx}
\usepackage[export]{adjustbox}
\usepackage{epsfig}
\usepackage{lineno,hyperref,appendix}
\usepackage{amsmath}
\usepackage{mathrsfs}
\usepackage{amsfonts}
\usepackage{color}
\usepackage{amssymb}
\usepackage{amsthm}
\usepackage{natbib} 
\biboptions{sort&compress}

\usepackage{lineno}
\numberwithin{equation}{section}
\newtheorem{theorem}{Theorem}[section]
\newtheorem{lemma}[theorem]{Lemma}
\newtheorem{proposition}[theorem]{Proposition}
\newtheorem{definition}[theorem]{Definition}

\newdefinition{remark}{Remark}
\newproof{pf}{Proof}

\newcounter{examp}

\newenvironment{examp}
  {\refstepcounter{examp}\par\noindent\textbf{Example \arabic{examp}:}\ \rmfamily}
  {\medskip}
  
\journal{ }

\DeclareMathOperator{\Real}{Re}
\DeclareMathOperator{\sign}{sign}

\newtheoremstyle{nonitalic}
{3pt}
{3pt}
{\normalfont}
{}
{\bfseries}
{}
{.5em}
{}

\theoremstyle{nonitalic}

\newcounter{hypothesis}

\newtheorem{hyp}[hypothesis]{}

\hypersetup{
pdftitle={Dynamics of a coupled nonlocal PDE-ODE system with spatial memory: well-posedness, stability, and bifurcation analysis},
pdfsubject={},
pdfauthor={Y. Salmaniw, D. Liu, J. Shi, H. Wang},
pdfkeywords={},
}

\newcommand{\as}[1]{\left\vert#1\right\vert}

\newcommand{\norm}[1]{\left\Vert#1\right\Vert}

\newcommand{\grad}{\nabla}

\newcommand{\dx}{{\rm d}x}
\newcommand{\dxi}{{\rm d}\xi}

\begin{document}
	
	\begin{frontmatter}
		
  \author[label6]{Yurij Salmaniw\fnref{label1}}
  \affiliation[label6]{organization={
Mathematical Institute, University of Oxford},
	city={ Oxford},
	postcode={OX2 6GG},
	country={UK}}

 \author[label2,label5]{Di Liu\fnref{label1}\corref{cor1}}
  \ead{dliu@hznu.edu.cn}
 \cortext[cor1]{Corresponding author}
 \affiliation[label2]{organization={School of Mathematics, Hangzhou Normal University},
 	city={Hangzhou},
 	postcode={311121},
 	country={PR China}}
   \affiliation[label5]{organization={School of Mathematics and Statistics, Central South University},
 	city={Changsha},
 	postcode={410083},
 	country={PR China}}
  \fntext[label1]{Co-first author.}

\author[label4]{Junping Shi}
  \affiliation[label4]{organization={Department of Mathematics, William $\&$ Mary},
 	city={Williamsburg},
 	postcode={VA 23187-8795},
 	country={USA}}

 \author[label3]{Hao Wang}
  \affiliation[label3]{organization={Department of Mathematical and Statistical Sciences, University of Alberta},
	city={Edmonton},
	postcode={AB T6G 2G1},
	country={Canada}}
  
  \title{Dynamics of a coupled nonlocal PDE-ODE system with spatial memory: well-posedness, stability, and bifurcation analysis}

\begin{abstract}
Nonlocal aggregation-diffusion models, when coupled with a spatial map, can capture cognitive and memory-based influences on animal movement and population-level patterns. In this work, we study a one-dimensional reaction-diffusion-aggregation system in which a population’s spatiotemporal dynamics are tightly linked to a separate, dynamically updating map. Depending on the local population density, the map amplifies and suppresses certain landscape regions and contributes to directed movement through a nonlocal spatial kernel.

After establishing the well-posedness of the coupled PDE-ODE system, we perform a linear stability analysis to identify critical aggregation strengths. We then perform a rigorous bifurcation analysis to determine the precise solution behavior at a steady state near these critical thresholds, deciding whether the bifurcation is sub- or supercritical and the stability of the emergent branch. Based on our analytical findings, we highlight several interesting biological consequences. First, we observe that whether the spatial map functions as attractive or repulsive depends precisely on the map's relative excitation rate versus adaptory rate: when the excitatory effect is larger (smaller) than the adaptatory effect, the map is attractive (repulsive). Second, in the absence of growth dynamics, populations can only form a single aggregate. Therefore, the presence of intraspecific competition is necessary to drive multi-peaked aggregations, reflecting higher-frequency spatial patterns. Finally, we show how subcritical bifurcations can trigger abrupt shifts in average population abundance, suggesting a tipping-point phenomenon in which moderate changes in movement parameters can cause a sudden population decline.
\end{abstract}

		
	\end{frontmatter}
	
\section{Introduction}
Aggregation-diffusion equations and systems have been used widely in the biological sciences, including cell-cell adhesion \cite{ButtenschoenHillen2021,Buttenschoen2018SpaceJump,falco2022local}, crowd dynamics \cite{muntean2014collective,MR1698215}, biological aggregations \cite{MR2257718} and ecology \cite{cantrell2003spatial}, appearing earlier in physics \cite{barends1997groundwater}, fluid dynamics \cite{paterson1981first}, chemical engineering \cite{bird2007transport}, and phase separation in materials science \cite{MR1453735}. Many such mean-field equations can be obtained from their direct connection with discrete dynamics described by stochastic differential equations \cite{Carrillo2020LongTime,kolokolnikov2013emergent, motsch2014heterophilious,pareschi2013interacting,MR3858403} or random walk processes \cite{potts2016territorial, PottsLewis2016}. More recently, movement affected by nonlocal detection of external stimuli has been incorporated into ecological movement models as a way to capture cognitive influences known to play an important role in determining animal movement behavior \cite{WangYurij2023JMB, FaganLewis2013EL, Lewis2021}.

We are interested in studying the evolution of a population $u(x,t)$ described by the following general nonlocal aggregation-diffusion equation with birth/death:
	\begin{equation}\label{1.1}
	u_t = \grad \cdot \left( d \grad u + \alpha  u \grad ( G * k) \right) + f(u).
	\end{equation}
    Population locomotion is described by two distinct forces: random exploratory movement, expressed by linear diffusion at constant rate $d>0$, and directed movement at rate $\alpha \in \mathbb{R}$ corresponding to a nonlocal potential $G*k(x,t)$ (see \eqref{1.3}). The function $f(u)$ is either identically zero or of logistic type; see Hypothesis \textbf{\ref{H1a}}. When population growth is absent, the total population is conserved, and we explore the influence of movement only. This assumes that the movement dynamics occur on a timescale much faster than the growth dynamics, as is sometimes the case for larger mammals \cite{PottsLewis2019}. With population growth included, we investigate the influence of growth dynamics paired with complex movement mechanisms. Essential to this perspective is the role of the spatial map $k(x,t)$, which we develop as follows.

\subsection{The spatial map}

    Often the quantity $k(x,t)$ is known \textit{a priori}, such as a resource distribution \cite{Fagan2017perceptual, cantrell2003spatial}, or may depend directly on the population density $u$ as in classical aggregation-diffusion equations, i.e., $k:= u$ as in \cite{Ducrot2014, DucrotFuMagal2018,Carrillo2020LongTime}; or more generally $k:= \Phi(u)$. In other cases, $k$ may itself evolve dynamically according to an auxiliary differential equation, as in the well-studied case of (chemo)taxis, for which $k$ describes the evolution of a (chemical) signal deposited by population $u$ \cite{Stevens1997, Hillen2008, Buttenschoen2018SpaceJump}. 

    	The present work considers the case where $k$ is modulated dynamically in response to the local density $u$ through a linear ordinary differential equation. The evolution of $k$ incorporates two opposing processes: an \textit{excitation} (or “positive encoding”) component described by $g_1(u)$, and an \textit{adaptation} (or “negative encoding”) component described by $g_2(u)$. Here, $g_1, \, g_2 : \mathbb{R}^+ \rightarrow \mathbb{R}^+$ are given functions which describe how members of the population $u$ encode a spatial location within the map $k$ in response to the (local) population density $u$, see Hypotheses \textbf{\ref{H2a}}-\textbf{\ref{H2b}}. The excitation process $g_1(\cdot)$ causes the map $k$ to increase; the adaptation process $g_2 (\cdot)$, on the other hand, causes the map to decrease. The balance of these two processes then determines the net effect on $k$, which is described by the following ordinary differential equation:
        \begin{equation}\label{1.2}
	k_t = g_1(u) - g_2(u) k.
	\end{equation}
    
    The form of \eqref{1.2} appears in several modeling contexts. In \cite{lewis1993modelling}, $k$ is used to describe the evolution of scent marks on the landscape deposited by wolves through Raised Leg Urination (RLU). In this case, $k$ is a map of scents distributed on the landscape, known to influence the movement of interacting wolf packs. While this form is primarily phenomenological, similar forms for $k$ are derived from a random walk process for some special cases of $g_i(\cdot)$, see \cite{potts2016territorial, PottsLewis2016}. In \cite{potts2016territorial}, $k$ is a spatial map that records areas (`conflict zones') of direct interaction (`conflicts') between members of differing populations. In \cite{ErbanOthmer2004}, an ordinary differential equation of the form \eqref{1.2} describes the internal state of an E. Coli cell, modulating how likely the cell is to reorient in space. This internal state is then paired with a transport equation describing the movement of cells. The form of \eqref{1.2} may also be of interest to models of neuroscience describing networks of neurons \cite{carrillo2022noise,carrillo2023noise}: it is now understood that there exist neurons, the so-called \textit{place cells} \cite{okeefe1971hippocampus} and \textit{grid cells} \cite{hafting2005microstructure}, which modulate their firing rate according to the location of the organism. Moreover, studies suggest the firing rate of visual cortical neurons can modulate according to ``attention" or whether stimulus appears alone or paired with distractors \cite{reynolds2004attentional}. In this context, equation \eqref{1.2} assumes a correlation between population density and firing rates of neurons so that regions in space are recorded or suppressed according to the excitatory and adaptatory properties of $g_1$ and $g_2$, respectively.

\subsection{Nonlocal effects}
    
Implicit in model \eqref{1.1} is an assumption that the entire population $u$ has access to the same information held within the spatial map $k$. This is observed across a wide array of taxa, from bacteria \cite{Bassler2002, Waters2005} to insects, birds, and larger mammals \cite{Couzin2009}. To improve the realism of our model, we assume that the population has access to the entire spatial map $k$, but has better access to information near the current point of occupation. Bias in the movement of the population is then given by the nonlocal potential $G*k$, where $G$ is a spatial kernel. For notational brevity, we will use the notation $\overline{k} (x,t)$ to denote the spatial convolution:
	\begin{equation}\label{1.3}
	\overline{k} (x,t) := G\ast k (x,t) = \int_\mathbb{R} G(x-y) k(y,t) dy,
	\end{equation}
	   Kernels appearing in nonlocal aggregation-diffusion equations can vary widely; we do not attempt to be comprehensive here. Prototypical kernels found in ecological applications include the top-hat kernel, the exponential kernel, or the Gaussian kernel, each of which can be obtained from a general exponential power distribution \cite{Fagan2017perceptual} (see also \cite{WangYurij2023JMB}). In the present work, we study the exponential detection function
	\begin{equation}\label{0}
		G(x) := \frac{1}{2R} e^{- |x|/ R},
	\end{equation}
	   where $R\geq0$ is the so-called \textit{perceptual radius} or \textit{sensing radius}, roughly describing the maximum distance at which landscape features can be identified  \cite{Fagan2017perceptual}. $G(\cdot)$ can also be interpreted as the probability of detecting a landscape feature at location $y$, while located at a point $x$, depending on the distance $|x-y|$: for a fixed distance $|x-y|$, increasing the perceptual radius corresponds to a relative increase in the probability of detection. In practice, the exponential detection function corresponds to the detection of, e.g., electromagnetic fields. For example, many marine animals can detect the electric and magnetic fields caused by electromagnetic survey measurements \cite{NYQVIST2020104888}. Exponential detection functions can also link sensory systems with spatial representation and large-scale natural navigation \cite{GevaNature}. From an analysis perspective, the exponential function possesses some desirable properties as it can be viewed as the fundamental solution to a particular elliptic PDE, see Proposition \ref{prop:equivalentsystem}.

\subsection{The coupled nonlocal PDE-ODE system}
    
	Combining equation \eqref{1.1} with equations \eqref{1.2}-\eqref{0}, we study the following coupled nonlocal PDE-ODE system in one spatial dimension:
	\begin{equation}\label{1}
	\begin{cases}
	u_t = d u_{xx} + \alpha ( u \overline{k}_{x} )_{x} + f(u), & x\in \mathbb{R}, \ t>0, \\
	k_t = g_1(u) - g_2(u) k, & x \in \mathbb{R}, \ t> 0, \\
	u(x,0)=u_0(x), k(x,0)=k_0(x),& x \in \mathbb{R}. \\
	\end{cases}
	\end{equation}
 Our goal is to describe the behaviour of solutions to \eqref{1} and its associated steady states depending on the encoding processes described by $g_i(\cdot)$, the population dynamics described by $f(\cdot)$, and the relative strength of diffusion and aggregation described by the parameters $d$ and $\alpha$, respectively. 

 The remainder of the paper is organized as follows. In Section \ref{sec:mainresults}, we state our key hypotheses and present our main results. In Section \ref{sec:interpretationofmainresults}, we explore some of the key biological insights gained from our analysis.
 Section \ref{sec:prelim} contains preliminary materials, notations, and notions of solution to the time-dependent problem \eqref{1}. Section \ref{sec:wellposed} is dedicated to the proof of well-posedness of the system, and we prove our bifurcation theorems in Section \ref{sec:biftheory}. To complement our theoretical insights, we explore some of the quantitative properties of solutions through numerical simulation in Section \ref{sec:apps}. Finally, in Section \ref{sec:nogrowth}, we briefly discuss some of the technical details behind the case with no growth dynamics. We move some technical but known arguments to the Appendix \ref{sec:appendix} to maintain a reasonable length of exposition while remaining self-contained.

\section{Main Results}\label{sec:mainresults}

\subsection{Hypotheses}\label{sec:hypotheses}
\noindent For simplicity, we will always assume the initial data are even, smooth functions:
\begin{align}\label{Hinitialdata}
    0 \lneqq u_0,\, k_0 \in C^3(\mathbb{R}) \cap L^2_{per} (\mathbb{R}).
\end{align}
We study those solutions that are even and periodic over $(-\pi,\pi)$. The periodicity and evenness ensure that the linearized problem has a simple eigenvalue and that the problem is defined on a compact interval so that the bifurcation theory of Crandall-Rabinowitz can be applied. Then, by considering solutions periodically extended over all of $\mathbb{R}$, the spatial convolution given by \eqref{1.3} is well-defined. This is a natural setting to study the problem due to technical challenges in analyzing such nonlocal models with a physical boundary \cite{WangYurij2023JMB}, which we do not explore further in the present work.

For notational brevity, we will denote the right-hand side of \eqref{1.2} as
$$
w(u,k):=g_1(u) - g_2(u) k,
$$ 
and denote by
$$
w_{k*}:=\frac{\partial w}{\partial k}\biggr\vert_{(u_*,k_*)}, \quad w_{u*} := \frac{\partial w}{\partial u} \biggr\vert_{(u_*,k_*)},
$$
where $(u_*,k_*)$ is the unique positive constant stationary state of \eqref{1} whose existence is guaranteed by Hypotheses \textbf{\ref{H1a}} and \textbf{\ref{H2a}} below. We assume that $f,g\in C^3(\mathbb{R})$ for simplicity. Our key hypotheses throughout the remainder of the manuscript are as follows.
\begin{hyp}\label{H1a}
		$f: \mathbb{R}^+ \rightarrow \mathbb{R}$ satisfies the following: $f(0)=f(u_*)=0$, \, $f(u) > 0$ for $0<u<u_*$, $f(u) < 0$ for $u_* < u < \infty$, \,  and $f_{u*}:=f_u(u_*)<0$.
\end{hyp}
\begin{hyp}\label{H1b}
  $f(u) \leq f^\prime(0) u$, \, and there exist constants $C>0$, $p \geq 1$ such that $\as{f(u)} \leq C (1 + u^p)$ for all $u \geq 0$.
\end{hyp}
\begin{hyp}\label{H2a}
		For $i=1,2$, $g_i : \mathbb{R}^+ \rightarrow \mathbb{R}^+$ satisfy $g_1(0) = 0 \leq g_2(0)$, and there holds \, $w(u_*,k_*)=0$.
\end{hyp}
\begin{hyp}\label{H2b} 
    There exists constants $C,M>0$, $p\geq1$, (possibly different from those of \textbf{\ref{H1b}}) such that $g_2(u) \leq C (1 + u^p)$ and $g_1(u) \leq M g_2 (u)$ for all $u \geq 0$.
\end{hyp}

\subsection{Statement of main results}\label{sec:statementofresults}

One of the primary tools used to obtain the results in this work is the following equivalence between solutions to the nonlocal system \eqref{1} and solutions to an auxiliary \textit{local} Neumann problem \eqref{1.18} when restricted to our particular class of solutions, i.e., those that are even and periodic functions of space. This is motivated by the approach used in, e.g., \cite{Britton1989,Britton1990SIAM,GourleyBritton1996} or \cite[Lemma 1]{ShiShiWang2021JMB}, where a nonlocal kernel of exponential type is featured, and their nonlocal problem is transformed to an equivalent local system.

\begin{proposition}[Equivalence to a local system]\label{prop:equivalentsystem}
    Assume $(u_0,k_0)$ are even functions satisfying \eqref{Hinitialdata}. The following equivalence holds. 
    \begin{enumerate}[i.)]
        \item Suppose $(u,k)$ is an even, $2\pi$-periodic classical solution solving system \eqref{1}. Then, $(U,K,V) := (u, k, G*k)$ restricted to the interval $(0,\pi)$ is a classical solution of the following local parabolic-ordinary-elliptic system 
\begin{equation}\label{1.18}
   \begin{cases}
U_t = d U_{xx} + \alpha ( U V_{x})_{x} + f(U), & x\in (0,\pi), \ t>0, \\
K_t = g_1(U) - g_2(U) K, & x \in (0,\pi), \ t> 0, \\
0 = V_{xx}-\frac{1}{R^2}(V-K), & x \in (0,\pi), \ t> 0,\\
U_x(0,t)=U_x(\pi,t)=0,\ V_x(0,t)=V_x(\pi,t)=0,\ & t>0,
\end{cases} 
\end{equation} 
with initial data $(U(x,0), K(x,0), V(x,0)) = (u_0,k_0,G*k_0)$.
\item Suppose $(U,K,V)$ is a classical solution solving the local Neumann problem \eqref{1.18} with initial data $(U(x,0), K(x,0), V(x,0)) = (u_0, k_0, G* k_0)$. Then, the pair $(u,k) := (U,K)$, reflected over $(-\pi,0)$ and periodically extended to $\mathbb{R}$ is an even, $2\pi$-periodic classical solution solving system \eqref{1}.
    \end{enumerate}
\end{proposition}

\begin{proposition}[Equivalence of stability between nonlocal and local systems]\label{prop:equivalenceofstability}
Suppose $(u_*,k_*)$ and $(u_*,k_*,v_*)$ are the constant steady states of system \eqref{1} and \eqref{1.18}, respectively. Then the constant steady state  $(u_*,k_*)$ of \eqref{1} has the same linear stability as the constant steady state $(u_*,k_*,v_*)$ with respect to \eqref{1.18}.
\end{proposition}

In \cite{Britton1989,Britton1990SIAM,GourleyBritton1996}, nonlocality in space, time, or both simultaneously is considered, where the nonlocality appears in the population dynamics describing nonlocal competition. This is much different from problem \eqref{1}, where the nonlocal potential $G*k$ appears at the highest order (i.e., within the advective term with first and second derivatives). In \cite{ShiShiWang2021JMB} the nonlocality appears at the level of aggregation terms, but the kernel is purely temporal, i.e., it is nonlocal in time only. In these works, special properties of the exponential kernel are used to convert the problem from a nonlocal system to a local one, giving a result similar to our Proposition \ref{prop:equivalentsystem}. Under this equivalence, we can prove the following well-posedness result. 

\begin{theorem}[Existence of unique classical solution]\label{thm:wellposed}
    Suppose hypotheses \textbf{\textup{\ref{H1b}}}-\textbf{\textup{\ref{H2b}}} hold. Then there exists a global weak solution $(u,k)$ solving problem \eqref{1} in the sense of Definition \ref{def:weaksoln}. Moreover, the solution is the unique, global classical solution in the sense of Definition \ref{def:classicalsolution}.
\end{theorem}
\noindent\textbf{Remark.} \textit{Note that our well-posedness result does not rely on Hypothesis \textbf{\textup{\ref{H1a}}} , and thus includes the no-growth case  $f(u) \equiv 0$.}\\

We observe that the value of $u_*$ is determined precisely by the growth dynamics $f(\cdot)$ under Hypothesis \textbf{\ref{H1a}}, otherwise it is determined precisely by the initial total density $\int_\Omega u_0 (x) {\rm d}x$ (see Section \ref{sec:nogrowth} for a full discussion of the case without population growth dynamics). We then have the following expression for $k_*$:
\begin{align}\label{q:kstar}
    k_* = \frac{g_1(u_*)}{g_2(u_*)} >0,
\end{align}
where the positivity follows from Hypothesis \textbf{\ref{H2a}}. Moreover, since $g_2(u_*)>0$, there always holds $w_{k*} = - g_2(u_*) < 0$, and so the sign of the quantity $w_{u*}$ becomes critical to our analysis of system \eqref{1}. It is then informative to compute
$w_{u*} = g_1 ^\prime (u_*) - g_2 ^\prime (u_*) k_*$ and substitute \eqref{q:kstar} for $k_*$ to find the following relation for $w_{u*}$:
\begin{align}\label{q:wustar}
w_{u*} = g_1 (u_*) \left( \frac{g_1 ^\prime(u_*)}{g_1(u_*)} - \frac{g_2^\prime(u_*)}{g_2(u_*)}  \right).
\end{align}
Therefore, the sign of $w_{u*}$ depends precisely on the difference between the \textit{relative rates of change} of the encoding processes $g_1$, $g_2$ near the constant state $u_*$.

To study the local stability of the constant state $(u_*,k_*,v_*)$ with respect to \eqref{1.18} for different $\alpha$, we first define the following quantity for a given wavenumber $n \in \mathbb{N}$:
\begin{equation}\label{1.13}
    \alpha_n(R) := \dfrac{(1+n^2 R^2)(d n^2 - f_{u*}) w_{k*}}{u_* w_{u*} n^2}.
\end{equation}
Notice that $\alpha_n(R)$ are either all positive or all negative depending on the sign of $w_{u*}$, which is drastically different from the situation for the top-hat detection function considered in \cite{liu2023biological}.
Depending on the sign of $w_{u*}$, we subsequently define the following critical parameters for any given $R>0$ fixed:
\begin{align}\label{alpha*}
    \begin{cases}
         \alpha_* (R):=\max\limits_{n\in \mathbb{N}}\alpha_n (R) <0\,  \text{ whenever }\, w_{u*}>0\,;\\
         \alpha^*(R):=\min\limits_{n\in \mathbb{N}}\alpha_n (R)
        >0\,  \text{ whenever }\, w_{u*}<0.
    \end{cases}
\end{align}
We then refer to the smallest value of $n$ at which the max or min is achieved the \textit{critical wavenumber}. The definition of \eqref{alpha*} is motivated by Lemma \ref{lem:1}, where concavity of $n^2 \mapsto \alpha_{n}(R)$ for fixed $R$ is obtained depending on the sign of $w_{u*}$. We first state the following local stability result.

\begin{theorem}[Local stability of constant states]\label{thm:1.2}
Fix $\alpha \in \mathbb{R} \setminus \{0\}$. Let $(u_*,k_*,v_*)$ be the unique positive constant steady state of \eqref{1.18}, $w_{u*}$ be as defined in \eqref{q:wustar} and $\alpha_n$ be as defined in \eqref{1.13}. The following hold.
	\begin{enumerate}
	    \item[(\romannumeral1)] $\lambda = 0$ is an eigenvalue of the characteristic equation of system \eqref{1.18} if and only if $\alpha=\alpha_n$.
	    \item[(\romannumeral2)] There are no purely imaginary eigenvalues of the characteristic equation of system \eqref{1.18} for any $\alpha\in {\mathbb R}$.
     \item[(\romannumeral3)] Suppose $w_{u*}=0$. Then $(u_*, k_*, v_*)$ is locally asymptotically stable with respect to \eqref{1.18}.
	    \item[(\romannumeral4)] Suppose $w_{u*}>0$ ($w_{u*}<0$). The following dichotomy is observed.
     \begin{enumerate}
         \item Any eigenvalue $\lambda$ of the characteristic equation of system \eqref{1.18} satisfies $\Real{(\lambda)} < 0$ whenever $\alpha>\alpha_*$ ($\alpha<\alpha^*$), and $(u_*, k_*, v_*)$ is locally asymptotically stable with respect to \eqref{1.18};
         \item There is at least one eigenvalue $\lambda$ of the characteristic equation of system \eqref{1.18} satisfying $\Real{(\lambda)} > 0$ whenever $\alpha<\alpha_*$ ($\alpha>\alpha^*$), and $(u_*, k_*, v_*)$ is unstable with respect to \eqref{1.18};
     \end{enumerate}
	\end{enumerate}
\end{theorem}

\noindent Below is a figure depicting the region of local stability in the $(R,\alpha)$-plane as presented in Theorem \ref{thm:1.2}.

\begin{figure}[htbp]
	\centering	
 
   \adjustbox{center}{%
    \includegraphics[width=0.75\textwidth]{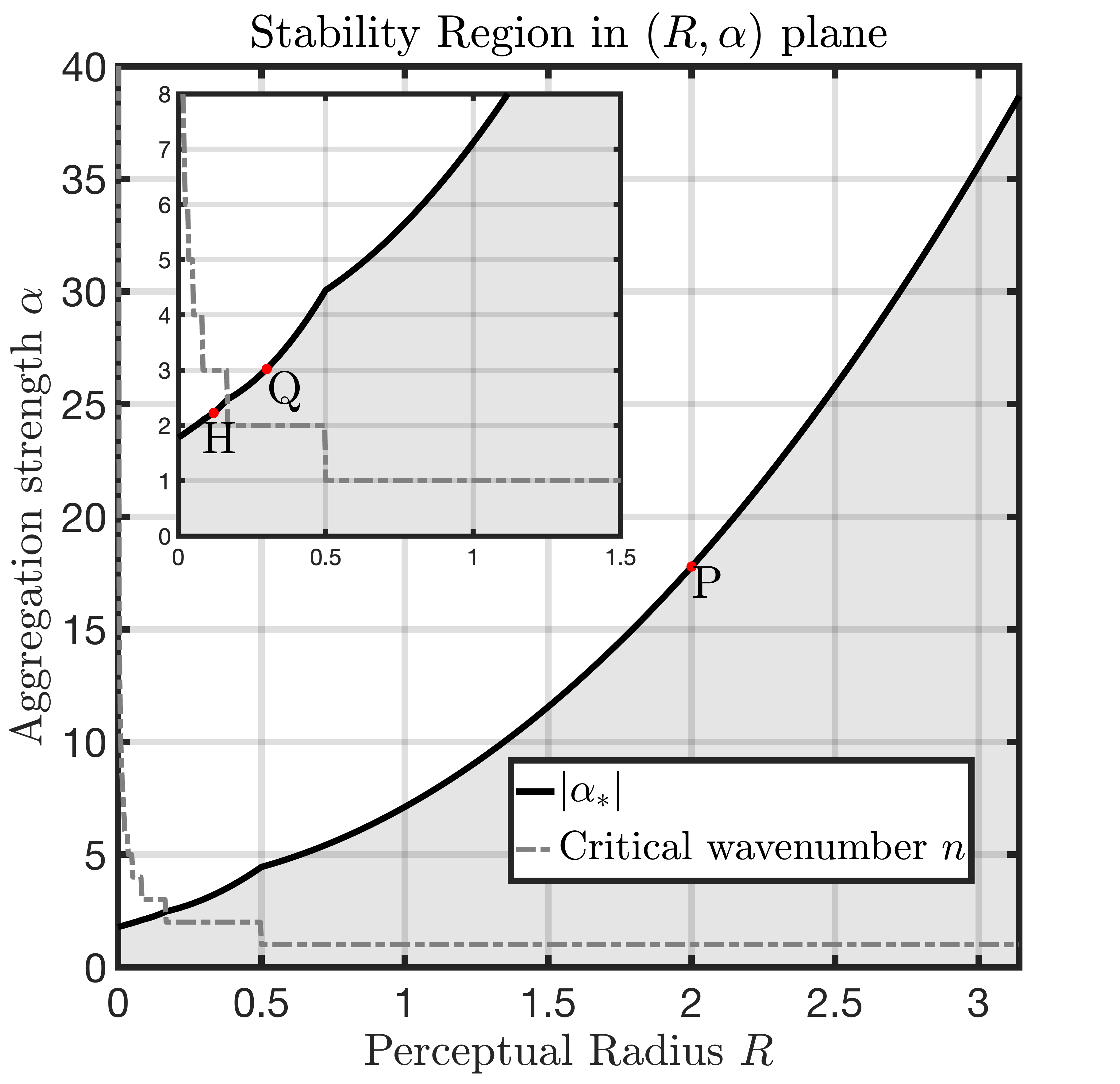}
  }
	\caption{A graphical depiction of the region of (local) stability of the homogeneous state in the $(R,\alpha)$-plane for system \eqref{1.1} with growth dynamics. The shaded region corresponds to local stability, while the white regions correspond to linear instability. The black dividing line is given by Equation \eqref{1.13}. The dashed line corresponds to the wavenumber at which destabilization occurs. The specific points $P=(2,17.7895)$, $Q=(0.3,3.0242)$, $H=(0.12,2.2328)$ are used in \textbf{Example} \ref{ex1} of Section \ref{1.13}; these are also the values chosen for Figures \ref{fig:3}, \ref{fig:4} and \ref{fig:2}.
 }
	\label{fig:1}
\end{figure}

We can then describe the quantitative properties of the non-constant steady state that appears near the critical bifurcation point.

\begin{theorem}[Description of bifurcations I]\label{thm:bif2}
Assume hypotheses \textbf{\textup{\ref{H1a}-\ref{H2b}}} hold and let $\alpha_n$ be as given in \eqref{1.13}. Then $\alpha=\alpha_n$ is a steady-state bifurcation point of system \eqref{1.18}. Moreover, near $(\alpha,u,k,v)=(\alpha_n,u_*,k_*,v_*)$, the bifurcating steady-state solutions from the line of constant solutions $\{(\alpha,u_*,k_*,v_*):\alpha\in\mathbb{R}\}$ lie on a smooth curve
\begin{equation}\label{1.22}
    \Gamma_n=\{(\alpha_n(s),u_n(s,\cdot),k_n(s,\cdot),v_n(s,\cdot)): -\delta<s<\delta\},
\end{equation}
with 
\begin{equation}
\begin{aligned}
    \alpha_n(s) = \alpha_n + \alpha_n'(0) s + z_{0,n}(s) s^2,\\
    \end{aligned}
\end{equation}
and 
\begin{equation}\label{sol}
\begin{aligned}
    \left(\begin{array}{c}
u_n(s,\cdot)
\\ 
k_n(s,\cdot)\\
v_n(s,\cdot)
\end{array}\right) = \left(\begin{array}{c}
u_*
\\ 
k_*\\
v_*
\end{array}\right) + s \left(\begin{array}{c}
1
\\ 
M_1\\
M_{2,n}
\end{array}\right) \cos(nx) + s^2 \left(\begin{array}{c}
z_{1,n}(s,\cdot)
\\ 
z_{2,n}(s,\cdot)\\
z_{3,n}(s,\cdot)
\end{array}\right)
\end{aligned}
\end{equation}
where $\delta>0$ is a constant, $\alpha_n'(0)=0$,
\begin{equation}\label{1.24}
    M_1=-\dfrac{w_{u*}}{w_{k*}},\ M_{2,n}=-\dfrac{w_{u*}}{w_{k*}}\left(\dfrac{1}{1+n^2 R^2}\right)=\left(\dfrac{1}{1+n^2 R^2}\right)M_1 ,
\end{equation}
and $z_{i,n}$ are smooth functions $z_{0,n}: (-\delta,\delta) \rightarrow \mathbb{R}$ and $z_{1,n}, z_{2,n}, z_{3,n}: (-\delta,\delta) \rightarrow X$ satisfying $z_{0,n}(0)=0$, $z_{i,n}(0,\cdot)=0$ $(i=1,2,3)$.
\end{theorem}

From Theorem \ref{thm:bif2}, we conclude that the bifurcation at $\alpha=\alpha_n$ is not a transcritical one but a pitchfork one since we find that $\alpha_n'(0)=0$. In Figures \ref{fig:3}-\ref{fig:4}, we display a bifurcation diagram depicting the result of Theorem \ref{thm:bif2} for two different values of $R$. We observe the pitchfork behavior in the first case ($R=0.3$). Moreover, we observe that the emergent stationary state has two peaks. Given Figure \ref{fig:1}, this is what we expect, where the point $Q$ predicts that the emergent state should appear with frequency $2$. Similarly, in the second case ($R=2.0$), we observe in Figure \ref{fig:4} that the emergent branch has a single peak, consistent with the prediction for point $P$ in Figure \ref{fig:1}. An essential difference, however, is the stability of the emergent branches.

Therefore, by the perturbation results in \cite{Crandall1973}, we obtain the following stability theorem, which determines whether the pitchfork bifurcation is forward or backward and the stability of the emergent branch, which fully resolves the differences observed between Figures \ref{fig:3} and \ref{fig:4}.

\begin{theorem}[Description of bifurcations II]\label{thm:bif3}
Let $(\alpha_n(s),u_n(s,\cdot),k_n(s,\cdot),v_n(s,\cdot))$ ($|s|<\delta$) be the bifurcating non-constant steady state solutions from the constant ones at $\alpha=\alpha_n$. Let $\alpha_n''(0)$ be as given in \eqref{alpha'}. Then we have the following results:
	\begin{enumerate}
	    \item[(\romannumeral1)] if $w_{u*}>0$, suppose that $\alpha_*=\alpha_M$ for some $M\in {\mathbb N}$, then the pitchfork bifurcation at $\alpha=\alpha_M$ is backward and the bifurcating solutions are locally asymptotically stable with respect to \eqref{1.18} if $\alpha_M''(0)<0$, and it is forward and the bifurcating solutions are unstable if   $\alpha_M''(0)>0$;
	    \item[(\romannumeral2)] if $w_{u*}<0$, suppose that $\alpha_*=\alpha_N$ for some $N\in {\mathbb N}$, then the pitchfork bifurcation at $\alpha=\alpha_N$ is forward and the bifurcating solutions are locally asymptotically stable with respect to \eqref{1.18} if $\alpha_N''(0)>0$, and it is backward and the bifurcating solutions are unstable if   $\alpha_N''(0)<0$.
	\end{enumerate}    
\end{theorem}

Based on Theorem \ref{thm:bif3}, $w_{u*}$, and $\alpha^{\prime \prime}(0)$, we can verify that in Figure \ref{fig:3}, the bifurcation is a supercritical one, and the emergent branch is (locally) stable. On the other hand, we can also verify that in Figure \ref{fig:4}, the bifurcation is subcritical, and the emergent branch is unstable. This is precisely what is found in Figure \ref{fig:4}, where the homogeneous state is locally stable near the critical threshold (yellow dots), but given a large enough perturbation, a stable high-amplitude state above the unstable branch is found (blue stars \& red pluses).

The following section discusses several biological insights emerging from our analysis.

\begin{figure}[htbp]
    \centering
    \includegraphics[trim={4cm 4cm 5cm 4cm},clip,width=0.75\linewidth]{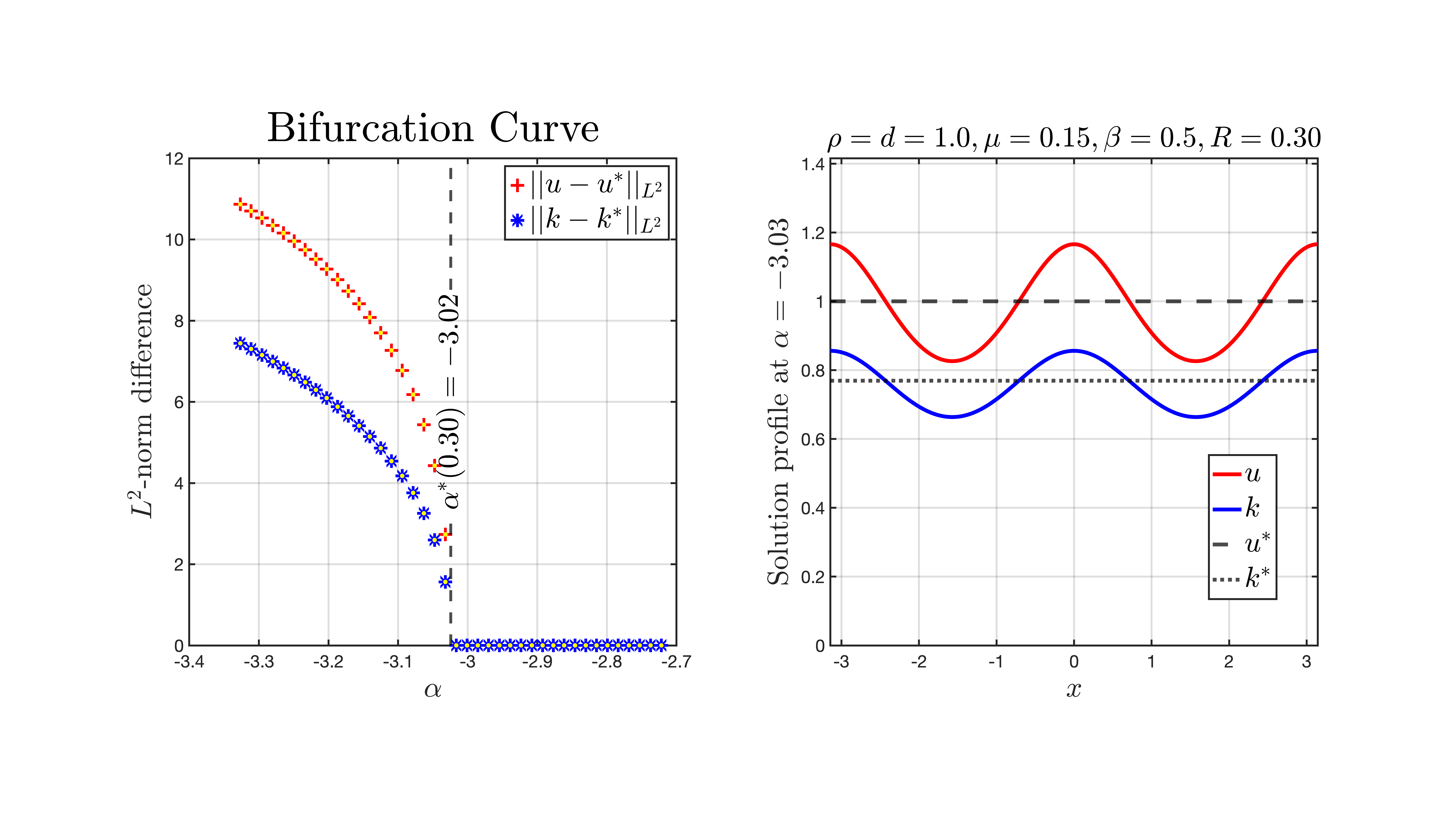}
    \caption{A bifurcation diagram near the critical threshold $\alpha^*$ when $R=0.3$ for \textbf{Example} \ref{ex1} of Section \ref{sec:apps} (left panel) and the solution profile at steady state just beyond the critical threshold (right panel).}
    \label{fig:3}
\end{figure}

\section{Interpretation of main results}\label{sec:interpretationofmainresults}

Beyond the technical insights gleaned from our detailed bifurcation analysis, we found several interesting consequences of these mathematical results that correspond directly to biological interpretation. We present them here and refer to Section \ref{sec:apps} for further technical details.

\subsection{Attractive vs. Repulsive Spatial Maps}

A defining feature of our model is that the sign of the derivative balance for the spatial map given by $w_{u*}$, defined in \eqref{q:wustar}, determines whether the map $k$ behaves as an effectively ``attractive" or ``repulsive" spatial map. For example, attractive maps may correspond to site fidelity \cite{Weinrich1998} or reinforcement learning \cite{Lewis2021}, while repulsive maps may correspond to avoiding areas too recently visited (``time since last visit") \cite{Schlagel2017} or overexploited sites. 

Concretely, recall from \eqref{1.13} that, for all inputs fixed, there is only one destabilization direction: pattern formation occurs for negative values of $\alpha$ \textit{xor} (in the exclusive sense) for positive values of $\alpha$. By \eqref{q:wustar}, the direction of destabilisation then depends on the relative magnitudes of $g^\prime (u^*)/g_1(u^*)$ and $g_2^\prime(u^*) / g_2(u^*)$, which are precisely the relative rates of change of the excitation and adaptation rates for the dynamics of the map $k$ at the homogeneous state $u^*$, respectively. When the relative excitation rate $g^\prime(u^*)/g_1(u^*)$ is the larger quantity, the map effectively acts as an \emph{attractive} potential: the population density $u$ is in phase with the peaks of the spatial map $k$, and the patterned state only emerges at sufficiently negative aggregation strengths $\alpha \ll 0$. Conversely, when the relative adaptation rate $g_2^\prime(u^*) / g_2(u^*)$ is the larger quantity, the map exerts a \emph{repulsive} effect: population peaks of $u$ become out of phase with the spatial map $k$, and patterned states occur only at sufficiently positive values $\alpha \gg 0$. In Figures \ref{fig:3}, \ref{fig:4}, and \ref{fig:2} (see \textbf{Example} \ref{ex1} of Section \ref{sec:examples}), we observe a patterned population profile in phase with the spatial map $k$. Conversely, in Figures \ref{fig:5}-\ref{fig:7} (see \textbf{Example} \ref{ex2} of Section \ref{sec:examples}), we now observe a patterned profile that is out of phase with the spatial map.

This mechanism parallels the classical scalar aggregation-diffusion equation on the torus \cite{Carrillo2020LongTime}, where the kernel itself is strictly attractive ($-G(x)$) or strictly repulsive ($+G(x)$). In our case, the kernel $G$ is \emph{fixed}, and it is \emph{the dynamics of the map ODE}—specifically, which of $g_1$ or $g_2$ dominates near $u^*$—that tips the system between attracting versus repelling regimes. This can be seen directly from Theorem \ref{thm:bif2}, where the phase relationship between $u$ and $k$ is determined by the sign of the coefficient $M_1$ (defined in \eqref{1.24}), whose sign is determined by $w_{u*}$. This highlights how excitatory and adaptive memory processes can reverse the population’s movement response toward desirable areas or away from undesirable areas, causing strikingly different spatial distributions depending on the relative excitatory versus inhibitory encoding strengths in the brain or sensory system. 

\begin{figure}[htbp]
    \centering
    \includegraphics[trim={4cm 4cm 5cm 4cm},clip,width=0.75\linewidth]{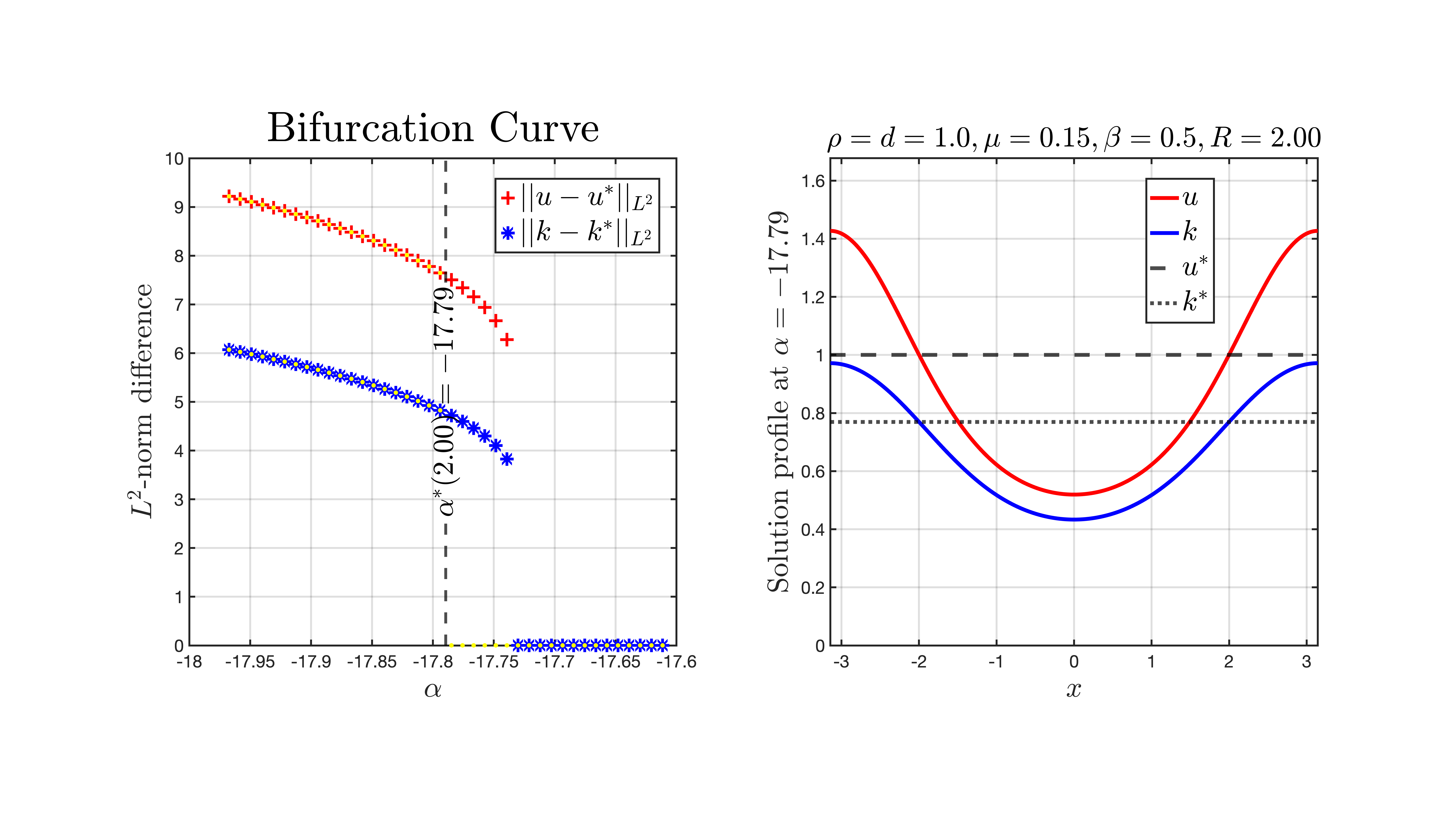}
    \caption{A bifurcation diagram near the critical threshold $\alpha^*$ when $R=2.0$ for \textbf{Example} \ref{ex1} of Section \ref{sec:apps} (left panel) and the solution profile at steady state just beyond the critical threshold (right panel).}
    \label{fig:4}
\end{figure}

\subsection{Competition Induces Higher Frequency Patterns} When growth dynamics are excluded (i.e., $f(u)\equiv 0$), we observe from \eqref{1.13} with $f_{u*}=0$ (see also \eqref{1.13'}) that the emergent patterned state will always occur at $n=1$, which is to say, the patterned state will always feature one single aggregate. In fact, in this case the critical aggregation strength $\alpha^* (R)$ depends quadratically on the perceptual radius $R$ given precisely as
$$
\alpha^*(R) = \frac{d w_{k*}}{u^* w_{u*}} (R^2 + 1).
$$
This is much different than the curve displayed in Figure \ref{fig:1}, where the frequency of the emergent patterned state increases as the perceptual radius $R$ decreases. Moreover, this increase in frequency holds regardless of whether the map is attractive or repulsive. Thus, more localised (i.e., stronger) aggregation drives the system toward higher-frequency modes. This is observed, for example, in Figures \ref{fig:3}, \ref{fig:4}, and \ref{fig:2}, where the number of peaks of the population density increases from one (Figure \ref{fig:4}) to two (Figure \ref{fig:3}) to three (Figure \ref{fig:2}) peaks as the radius $R$ decreases. From a biological standpoint, this can be understood intuitively: in the presence of density-limiting competition, individuals that cluster too tightly suffer reduced per-capita growth, pushing the system to form smaller, more numerous clumps rather than a single large aggregation. Consequently, one obtains a richer variety of spatial structures when a dynamic spatial map and local competition occur simultaneously. This underscores that competition is necessary to generate multi-peaked configurations in this framework.

\subsection{Impact of aggregation on abundance and the role of sub- versus supercritical bifurcations}

A core difference from standard aggregation-diffusion problems studied on the torus is that, here, both \emph{forward} (supercritical) and \emph{backward} (subcritical) bifurcations can arise (see Theorem \ref{thm:bif3} and Figures \ref{fig:4} and \ref{fig:5}). In classical periodic aggregation-diffusion equations, the homogeneous steady state typically loses stability through a supercritical pitchfork bifurcation; see, for example, \cite[Theorem 4.2]{Carrillo2020LongTime}. By contrast, our combined PDE-ODE system also supports subcritical pitchfork bifurcation behaviour. This is an unexpected insight on its own, but it is of more consequence when also considering changes in population abundance. We explain this as follows.

\begin{figure}
    \centering
    \includegraphics[width=\linewidth]{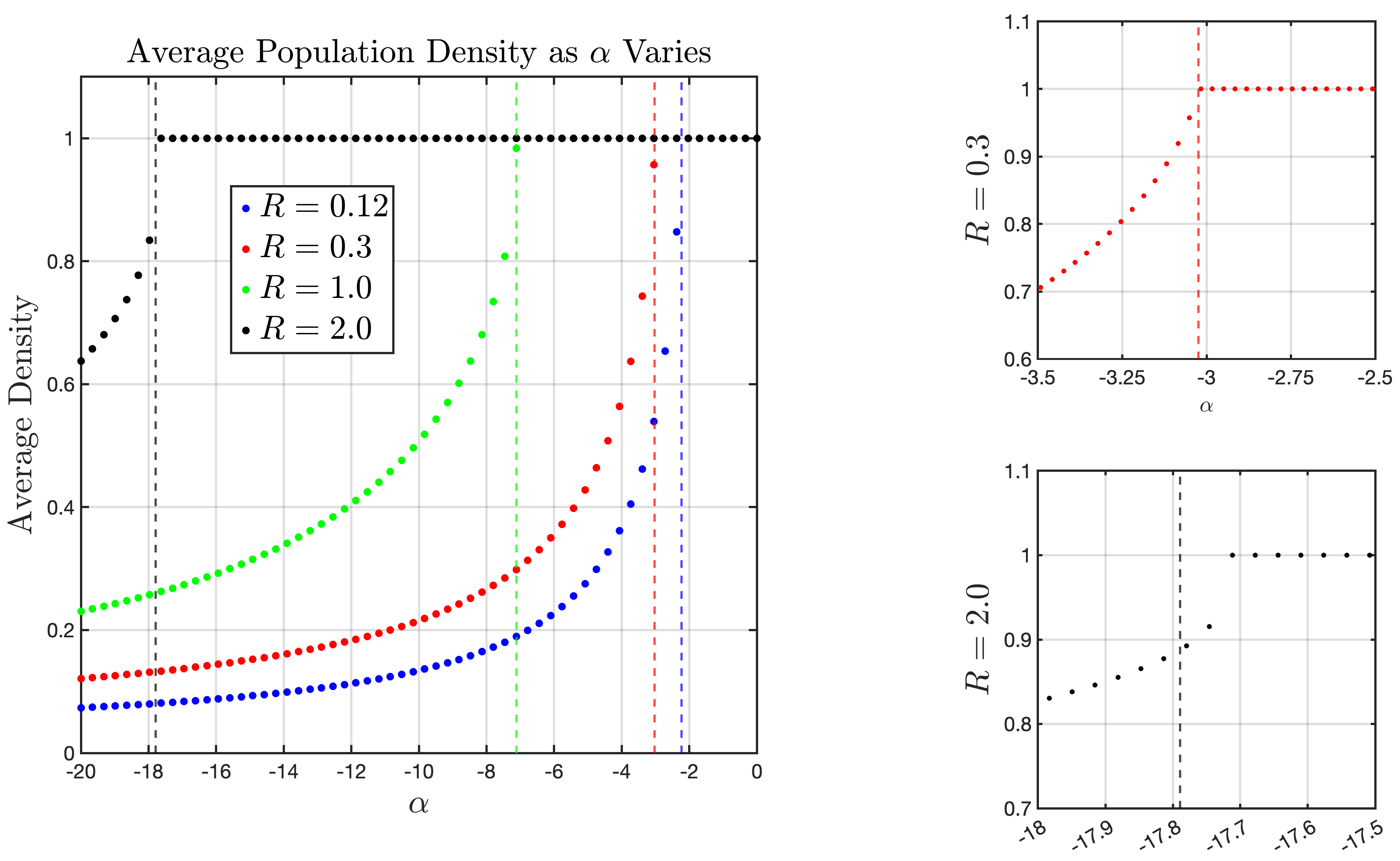}
    \caption{The average population density as a function of aggregation strength $\alpha$ for four $R$ values (colored dots). The vertical dashed lines correspond to the critical threshold $\alpha^*(R)$ obtained from a linear stability analysis (see \eqref{alpha*}). The right panels depict a zoomed-in version of the $R=0.3$ case (red dots) and the $R=2.0$ case (black dots), exemplifying the influence of subcritical versus supercritical bifurcation structure and the impact on the average population density.}
    \label{fig:totalmass}
\end{figure}

It is easy to see that the total population is conserved when $f(u) \equiv 0$, and without loss of generality, we may assume that the average density at this state is $1$. With population dynamics included, however, we have observed numerically that as the aggregation strength passes through the critical value, the abundance \textit{decreases}. This is depicted in Figure \ref{fig:totalmass}, where we display the average population density as a function of aggregation strength $\alpha$. We observe that for any perceptual radius $R>0$, the average density is non-increasing and is strictly decreasing beyond the critical threshold.

This may be understood intuitively: aggregates naturally include regions with densities above and below the density of a homogeneous state. To see this, integrate the equation for $u$ over $(-\pi,\pi)$ and apply the periodicity of solutions to find that
\begin{align}
    \int_\Omega u(1-u) \dx = 0.
\end{align}
Therefore, for any non-constant stationary state $u=u^*$, there must be regions of $\Omega$ with $u^*>1$ and other regions with $u^* < 1$. The population is then penalized for aggregating too closely, as our model does not explicitly include any benefit to forming localized aggregates (e.g., through group defense). 

When the bifurcation is \textit{supercritical}, we observe a smooth, decreasing transition from the average density of $1$ at the homogeneous state to a lower average density at a patterned state as the aggregation strength increases and the population aggregates become more localised. Referring again to Figure \ref{fig:3}, we observe that the smoothness of the transition from a homogeneous state to a patterned state translates to the smooth transition from high average density to lower average density (red dots of Figure \ref{fig:totalmass}).

However, the emergent low-amplitude state is unstable when the bifurcation is \textit{subcritical}. Consequently, we observe a discontinuous jump from the average density at the homogeneous state to a much smaller average density at the patterned state; moreover, this discontinuous transition occurs earlier than that of the supercritical case (the black dots of Figure \ref{fig:totalmass}). Referring now to Figure \ref{fig:3}, we see that this discontinuous transition in the average density corresponds directly to the discontinuity in the stability of the homogeneous state.

Summarizing, our numerical exploration suggests that once a patterned branch is triggered just past the critical aggregation strength $\alpha^*$, the average density may collapse to a value significantly below the homogeneous level, reminiscent of a tipping point phenomena in ecology \cite{Dakos2019}. This difference has implications for population management: certain parameter shifts (e.g., in $\alpha$ or the $g_1$. $g_2$ encoding rates) may abruptly plunge the population’s overall abundance into a lower state—highlighting a potential for sudden collapse or redistribution depending on the nature of the bifurcation. Crucially, this transition can be initiated by a strong disturbance (i.e., a large perturbation) of the current population density.

\section{Preliminaries}\label{sec:prelim}

Throughout the manuscript we will often write $\Omega$ to mean either $\mathbb{R}$, $(-\pi,\pi)$ or $(0,\pi)$ where the context will be made clear. We then denote $Q_T := \Omega \times (0,T)$ for $T>0$ fixed. For $1 \leq p \leq \infty$, $L^p (\Omega)$ denotes the usual Lebesgue space of $p$-integrable functions over the spatial domain $\Omega$; when $p=\infty$, this is the space of essentially bounded functions with $\norm{u}_{L^\infty(\Omega)}$ the essential supremum of $u$ over $\Omega$. Given a Banach space $X$ and $1 \leq q \leq\infty$, we denote by $L^q(0,T; X)$ the space of functions $u(\cdot,t) \in X$ with finite $L^q$-norm in time, i.e., $\norm{u}_{L^q (0,T; X)}^q = \int_0 ^T \norm{u(\cdot,t)}_X {\rm d}t < \infty$. In particular, for $1 \leq p,q < \infty$, we denote by $L^q ( 0,T; L^p (\Omega) )$ the Lebesgue space over $Q_T$ with norm
\begin{align}
    \norm{v}_{L^q (0,T; L^p (\Omega) )} := \left( \int_0 ^T \norm{v (\cdot,t)}_{L^p (\Omega)} ^q {\rm d}t \right)^{1/q},
\end{align}
and when $1 \leq p \leq q=\infty$ 
\begin{align}
    \norm{v}_{L^\infty (0,T; L^p (\Omega) )} := \sup_{t \in (0,T)} \norm{v(\cdot,t)}_{L^p(\Omega)}.
\end{align}

In cases where $p=q$ we simply write $L^p(Q_T) := L^p(0,T; L^p(\Omega))$. We denote by $C^{2+\sigma,1+\sigma/2} (Q_T)$ the class a functions twice differentiable in space and once differentiable in time (in the classical sense), with its second order spatial derivatives and first order time derivative being H{\"o}lder continuous for some exponent $\sigma \in (0,1)$. 

When performing the bifurcation analysis, we will denote by
\begin{equation}
    X = \{u \in H^2(0,\pi) : u'(0)=u'(\pi)=0 \}\quad \text{and}\quad Y = L^2(0,\pi),
\end{equation}
and define the complexiflication of a linear space $ Z $ to be $ Z_{\mathbb{C}}:=Z\oplus iZ=\{x_1+ix_2|x_1,x_2\in Z\} $. We will denote by $ \mathcal{N}(L) $ and $ \mathcal{R}(L) $ the kernel and the range of an operator $L$, respectively. $\sigma(L)$ and $\sigma_p(L)$ denote the spectrum and point spectrum of an operator $L$, respectively. For the complex-valued Hilbert space $Y_{\mathbb{C}}^2$, we use the standard inner product $ \langle u,v\rangle=\int_{0}^{\pi} \overline{u}(x)^{\rm T}v(x){\rm d}x $.\\

We understand a \textit{weak solution} to problem \eqref{1} in the following sense.
\begin{definition}[weak solution]\label{def:weaksoln}
    We call $( u, k )$ a \textup{weak solution} to problem \eqref{1} corresponding to the initial data $(u_{0}, k_{0})$ satisfying \eqref{Hinitialdata} if 
    \begin{align*}
    &u \in L^2(0,T; H^1 (\Omega)),\quad u (G * k_x) \in L^2(0,T; L^2(\Omega)), \\
    &k \in L^2(0,T; L^2(\Omega)),\quad u_t,\, k_t \in L^2(0,T; L^2 (\Omega)),
    \end{align*}
    and for all test functions $\phi,\, \psi \in L^2 (0,T; H^1 (\Omega))$ there holds
\begin{align}\label{defn:weaksolution}
    &\iint_{Q_T}  u_t \phi \dx {\rm d}t + d \iint_{Q_T} u_x \phi_x {\rm{d}}x {\rm{d}} t + \alpha \iint_{Q_T} u ( G * k )_x  \phi_x {\rm{d}}x {\rm{d}}t = \iint_{Q_T} f(u) \phi {\rm{d}}x {\rm{d}}t, \nonumber \\
    &\iint_{Q_T} k_t  \psi \dx {\rm d}t = \iint_{Q_T} w(u,k) \psi {\rm d}x {\rm d}t,
\end{align}
with the initial data satisfied in the sense of $L^2(\Omega)$. We call $(u,k)$ a \textup{global weak solution} if it is a weak solution for any $T>0$.
\end{definition}

When we refer to a \textit{strong solution} to problem \eqref{1}, we mean a bounded (i.e., $L^\infty(Q_T)$) weak solution $(u,k)$ corresponding to the initial data $(u_0,k_0)$ such that $u_x \in L^\infty(0,T; L^2(\Omega))$, $u_{xx} \in L^2(Q_T)$, and $u,\, k \in H^1(0,T; L^2(\Omega))$ with the initial data satisfied in the sense of $L^2(\Omega)$ (i.e., $u_t$, $k_t$ are now identifiable in a weak sense rather than in a dual sense). \\

For the equivalent local problem \eqref{1.18}, the definition of a weak solution is of the same flavor, instead with the understanding that $u (G* v)_x,\, v \in L^2(Q_T)$ and $v \in L^2(0,T; H^1( \Omega))$. Similarly, for strong solutions, we understand that $v \in L^2(0,T; H^2(\Omega))$. These are, in some sense, the minimal requirements to make sense of the two relations of \eqref{defn:weaksolution} and the counterpart for the local problem \eqref{1.18} after transformation. Notice carefully that we do not require weak differentiability of $k$ in this weak form; instead, the differentiation can be done (in a weak sense) with respect to the kernel $G$. \\

Finally, in the following sense, we refer to a \textit{classical} solution to problem \eqref{1}.
\begin{definition}[classical solution]\label{def:classicalsolution}
    We call $( u, k )$ a \textup{classical solution} to problem \eqref{1} corresponding to the initial data $( u_0, k_0 )$ if $u,\, k \in C^{2+\sigma, 1+\sigma/2} (\overline{Q}_T)$
    for some $\sigma \in (0,1)$ and satisfies the equation pointwise in $Q_T$. We call the classical solution \textup{global} if it is a classical solution for all $T>0$. 
\end{definition}

\section{Equivalence between nonlocal and local systems}\label{sec:equivalence}
\noindent In our subsequent analysis, it is essential that the local system \eqref{1.18} inherits the same solutions and stability properties as the nonlocal system \eqref{1}. In this section, we prove Proposition \ref{prop:equivalenceofstability}. The proof of Proposition \ref{prop:equivalentsystem} is included in Appendix \ref{subsec:A1}.

The linearized system of \eqref{1} at the constant steady state $(u_*,k_*)$ is given by:
\begin{equation}\label{20}
\left(\begin{array}{c}
\widetilde{u}_t \\ \widetilde{k}_t
\end{array} \right)=\mathcal{L_*}\left(\begin{array}{c}
\widetilde{u} \\ \widetilde{k}
\end{array} \right):=\left(\begin{array}{cc}
d \widetilde{u}_{xx} +f_{u*}\widetilde{u} + \alpha u_*(G\ast\widetilde{k})_{xx} \\ 
w_{u*}\widetilde{u} + w_{k*}\widetilde{k}
\end{array} \right),
\end{equation}
for $(x,t) \in \mathbb{R} \times (0,\infty)$. Substituting $\widetilde{u}=e^{\mu t + {\rm i} n x}\widehat{u}$, $\widetilde{k}=e^{\mu t + {\rm i} n x}\widehat{k}$,  (where $\mu\in \mathbb{C}, n\in \mathbb{N}$) into \eqref{20} and dropping the hats, we have 
\begin{equation}
\mu\left(\begin{array}{c}
u \\ k 
\end{array} \right)
=\left(\begin{array}{cc}
-d n^2 u  +f_{u*}u - \alpha u_* n^2C_n(G)k \\ 
w_{u*}u + w_{k*}k
\end{array} \right),
\end{equation}
where $C_n(G)=\int_{-\infty}^{+\infty}e^{-{\rm i}ny}G(y){\rm d}y=\frac{1}{1+n^2 R^2}$ is the $n^{\textup{th}}$ Fourier coefficient of the kernel $G$. Then we have the following lemma.
\begin{lemma}\label{lem:spectral1}
    Suppose $(u_*,k_*)$ is the constant steady state of system \eqref{1}. Let $\mathcal{L}_*$ be linearized operator as given in \eqref{20}. Then 
    \begin{equation*}
	\sigma(\mathcal{L_*})=\sigma_p(\mathcal{L_*})=\{\mu_n^{\pm}:n\in\mathbb{N}\cup\{0\}\}\bigcup\{w_{k*}\},
	\end{equation*}
    where 
	\begin{equation}\label{BC}
	\begin{split}
	\mu_n^{\pm}&=\dfrac{B(n^2)\pm\sqrt{B(n^2)^2-4C(\alpha,n^2)}}{2},\\
	 B(n^2)&=f_{u*}+w_{k*}-dn^2,\\
  C(\alpha,n^2)&=f_{u*}w_{k*}+\left(\alpha u_* w_{u*}\frac{1}{1+n^2 R^2}-dw_{k*}\right)n^2.
	\end{split}
	\end{equation}
\end{lemma}
\begin{proof}
    The proof is similar to the proof of \cite[Theorem 3.6]{liu2023biological} and so we omit it.
\end{proof}

Similarly, linearizing the system \eqref{1.18} about $(u,k,v)=(u_*,k_*,v_*)$ yields

\begin{equation}\label{J*}
    \left(\begin{array}{c}
\widetilde{u}_t
\\ 
\widetilde{k}_t\\
0
\end{array}\right)=J_*\left(\begin{array}{c}
\widetilde{u}
\\ 
\widetilde{k}\\
\widetilde{v}
\end{array}\right):=
\left(\begin{array}{c}
d \widetilde{u}_{xx} + \alpha_n u_* \widetilde{v}_{xx}  + f_{u*}\widetilde{u}
\\ 
w_{u*} \widetilde{u} + w_{k*} \widetilde{k}
\\
\widetilde{v}_{xx}-\frac{1}{R^2}(\widetilde{v}-\widetilde{k})
\end{array}\right)
\end{equation}
for $(x,t) \in(0,\pi) \times (0,\infty)$. For $\lambda\in \mathbb{C}$, $n\in \mathbb{R}$, we substitute $\widetilde{u}=e^{\lambda t + {\rm i} n x}\widehat{u}$, $\widetilde{k}=e^{\lambda t + {\rm i} n x}\widehat{k}$, $\widetilde{v}=e^{\lambda t + {\rm i} n x}\widehat{v}$ into \eqref{J*} and drop the hats to obtain the characteristic equation of \eqref{1.11} given by
\begin{equation}\label{1.11}
	\lambda \left(\begin{array}{c} 
	u \\ k \\ 0
	\end{array} \right) = \left(\begin{array}{ccc}
	 -d n^2 + f_{u*} & 0 &   -\alpha u_* n^2 \\ 
	w_{u*} &   w_{k*} & 0\\
	0 & \dfrac{1}{R^2} & -n^2 - \dfrac{1}{R^2}
	\end{array} \right) \left(\begin{array}{c} 
	u \\ k \\ v
	\end{array} \right).
	\end{equation}
Similar to Lemma \ref{lem:spectral1} we obtain the following.
\begin{lemma}\label{lem:spectral2}
    Suppose $(u_*,k_*,v_*)$ is the constant steady state of system \eqref{1.18}. Let $J_*$ be linearized operator as given in \eqref{J*}. Then 
    \begin{equation*}
	\sigma(J_*)=\sigma_p(J_*)=\{\lambda_n^{\pm}:n\in \mathbb{N}\cup\{0\} \}\bigcup\{w_{k*}\},
	\end{equation*}
    where 
	\begin{equation}\label{TD}
	\begin{split}
	\lambda_n^{\pm}&=\dfrac{T(n^2)\pm\sqrt{T(n^2)^2-4D(\alpha,n^2)}}{2},\\
	  T(n^2) &=-(d n^2 -f_{u*} - w_{k*})<0,\\
    D(\alpha, n^2) &= - (dn^2 - f_{u*})w_{k*} + \alpha u_* w_{u*} \dfrac{n^2}{1+n^2 R^2}.
	\end{split}
	\end{equation}
\end{lemma}
We now prove Proposition \ref{prop:equivalenceofstability}.
\begin{proof}[Proof of Proposition \ref{prop:equivalenceofstability}]
Combining Lemma \ref{lem:spectral1} and Lemma \ref{lem:spectral2}, we immediately observe that in either case the linear stability of the constant steady state is determined by $T(n^2)=B(n^2)$ and  $D(\alpha, n^2)=C(\alpha, n^2)$, and so the spectrum of the linearized operator of \eqref{1} is identical to the one of the linearized operator of \eqref{1.18}.
\end{proof}

\section{Well-posedness}\label{sec:wellposed}

In this section, we establish the existence and uniqueness of classical solutions to problem \eqref{1} in the class of periodically extended solutions from $[-\pi,\pi]$ to $\mathbb{R}$. Our approach is to work with the equivalent local problem \eqref{1.18}. We first show that any strong solution to system \eqref{1.18} is necessarily unique. To obtain the existence of a weak solution, we follow the approach of \cite{liu2023biological}: estimates on the density $u$ and the auxiliary function $v$ will follow from uniform estimates on the map $k$. Once a weak solution is obtained, we can improve the integrability of the solution and show that it is strong. Using Proposition \ref{prop:equivalentsystem}, we then have the existence of a strong solution to the original problem \eqref{1}, which we then show to be classical.

\subsection{Uniqueness}

We begin with the following uniqueness result for strong solutions of the local problem \eqref{1.18}.

\begin{lemma}[Uniqueness of smooth solutions]\label{lem:uniqueness}
    There is at most one smooth solution $(u,k)$ (in the sense of Definition \ref{def:classicalsolution}), periodic on $\mathbb{R}$, solving problem \eqref{1} corresponding to the initial data $(u_0, k_0)$.
\end{lemma}
The proof is standard, using properties of the strong solution, energy estimates, and Gr{\"o}nwall's lemma. In fact, the uniqueness holds for any periodic solution of \eqref{1}. The technical details are provided in Appendix \ref{sec:AppUniqueness}.

\subsection{A priori estimates}

The following result provides useful estimates on the quantity $k$, essentially identical to \cite[Lemma 2.4]{liu2023biological}.
\begin{lemma}\label{ODEbounds1}
Suppose $0 \lneqq w(x,t) \in C^{1,1} (\overline{Q}_T) \cap L^{1,1} (Q_T)$ satisfies a homogeneous Neumann boundary condition in $(0,\pi)$ for each $t \in (0,T)$. For each $x \in (0,\pi)$, let $k(x,\cdot)$ solve the ordinary differential equation
\begin{equation}\label{kProto}
\frac{d k}{d t} = g_1(w) - g_2(w) k
\end{equation}
where $g_1,\, g_2$ satisfy hypothesis \textbf{\textup{\ref{H2a}}}-\textbf{\textup{\ref{H2b}}} and $k(x,0) = k_0 (x)$ is even and satisfies \eqref{Hinitialdata}. Then there holds
\begin{equation}\label{kLinfest}
\sup_{t \in [0,T]} \norm{k (\cdot,t)}_{L^\infty (\Omega)} \leq M + \norm{k_0}_{L^\infty (\Omega)}
\end{equation}
\end{lemma}
\begin{proof}
By solving the ODE for each $x \in (0,\pi)$ fixed, it is not difficult to verify that if $w$ and $k_0$ satisfy the homogeneous Neumann boundary condition for each $t \in (0,T)$, then $k$ does as well. The rest of the proof follows from \cite[Lemma 2.4]{liu2023biological}. 
\end{proof}

We then establish the following \textit{a priori} estimates for any smooth solution $(u,k,v)$ solving problem \eqref{1.18}.
\begin{lemma}\label{lem:aprioriestimates1}
    Fix $2 \leq p < \infty$ and suppose $(u,k,v)$ is a smooth, nonnegative solution solving problem \eqref{1.18}. Then, under hypotheses \textbf{\textup{\ref{H1b}}}-\textbf{\textup{\ref{H2b}}} there holds
    \begin{align}
        \sup_{t \in (0,T)}\norm{v(\cdot,t)}_{W^{2,p} (\Omega)} \leq C ; \label{est:vW2p} \\
        \norm{u}_{L^\infty(Q_T)} + \sup_{t \in (0,T)} \norm{ u_x (\cdot,t)}_{L^2 (\Omega)} + \norm{ u_{xx}}_{L^2(Q_T)} + \norm{u_t}_{L^2(Q_T)}\leq C; \label{est:uW212} \\
         \norm{k_t(\cdot,t)}_{L^\infty(Q_T)} \leq C \label{est:ktLinf},
    \end{align}
    where $C = C(d, \alpha,R, f^\prime(0), M, p, \norm{u_0}_{L^\infty(\Omega)},\norm{k_0}_{L^\infty(\Omega)})$ is a uniform constant.
\end{lemma}
\begin{proof}
First, under hypotheses \textbf{\textup{\ref{H2a}-\ref{H2b}}} we have that $\norm{k}_{L^\infty(Q_T)} \leq M + \norm{k_0}_{L^\infty (\Omega)}$ by Lemma \ref{ODEbounds1}. 

Next, we note that since $k$ is nonnegative, $v$ is also nonnegative by definition. For $p \geq 2$, we multiply the equation for $v$ by $v^{p-1}$, integrate by parts and apply H{\"o}lder's inequality to find for each fixed $t \in (0,T)$ that
    \begin{align}
        \int_\Omega v^p \dx &= \int_\Omega k v^{p-1} \dx - R^2 (p-1) \int_\Omega v^{p-2} \as{v_x} \dx \leq \norm{k(\cdot,t)}_{L^p (\Omega)} \norm{v(\cdot,t)}_{L^p(\Omega)}^{p-1}.
    \end{align}
    Upon rearrangement we find that for each $t \in (0,T)$ fixed, $\norm{v(\cdot,t)}_{L^p(\Omega)} \leq \norm{k(\cdot,t)}_{L^p(\Omega)}$. Taking $p \to \infty$, applying estimate \eqref{kLinfest} and taking the supremum over $(0,T)$ leaves
    $$
    \norm{v}_{L^\infty(Q_T)} \leq M + \norm{k_0}_{L^\infty(\Omega)}.
    $$ By $L^p$-estimates for strong solutions of second order elliptic equations (see, e.g., \cite[Theorem 1]{Wang2003Geometric}), we conclude that for each $t \in (0,T)$ fixed there holds
    \begin{align}
        \norm{v_{xx}(\cdot,t)}_{L^{p} (\Omega)} \leq&\, C_1 \left( \norm{v(\cdot,t)}_{L^p(\Omega)} + \norm{k (\cdot,t)}_{L^p(\Omega)} \right) \nonumber \\
        \leq&\, 2 C_1 ( M + \norm{k_0}_{L^\infty(\Omega)} ),
    \end{align}
    where $C_1$ depends only on $R$, $\as{\Omega}$ and $p$. Paired with the Gagliardo-Nirenberg interpolation inequality, we obtain estimate \eqref{est:vW2p}.
    
    Since $v(\cdot,t) \in  W^{2,2}(\Omega)$, we apply the Sobolev embedding to $v_x(\cdot,t) \in H^1 (\Omega)$ for each $t \in (0,T)$ fixed to find
    $$
\norm{v_x (\cdot,t)}_{C^\sigma (\Omega)} \leq C_2 \norm{v_x (\cdot,t)}_{H^1(\Omega)}\leq C_1 (M + \norm{k_0}_{L^\infty(\Omega)}) =: C_2 ,
    $$
    for some $\sigma \in (0,1)$, and so $v_x \in L^\infty(Q_T)$. 

We now obtain estimates on $u$. To this end, we take the derivative of $\tfrac{1}{p}\norm{u(\cdot,t)}_{L^p (\Omega)}^p$ with respect to time and integrate by parts to obtain
\begin{align}\label{est:1.9}
    \frac{1}{p}\frac{{\rm d}}{{\rm d}t} \int_\Omega u^p \dx =& - (p-1) \int_\Omega u^{p-2} \as{u_x}^2 \dxi \nonumber \\
    -&\, \alpha (p-1) \int_\Omega u^{p-1} u_x v_x \dx + \int_\Omega (u) ^{p-1} f(u) \dx .
\end{align}
Cauchy's inequality with $\varepsilon = (\alpha C_2)^{-1}$ yields
\begin{align}
    \as{\alpha} u^{p-1} \as{u_x}\as{ v_x} \leq \frac{1}{2} u^{p-2} \as{u_x}^2 + \frac{(\alpha C_2)^2}{2} u^p .
\end{align}
By hypothesis \textbf{\textup{\ref{H1b}}} there holds $f(u) \leq f^\prime (0) u$ and so \eqref{est:1.9} becomes
\begin{align}\label{est:1.11}
    \frac{1}{p} \frac{{\rm d}}{{\rm d}t} \int_\Omega u^p \dx \leq& -\frac{(p-1)}{2} \int_\Omega u^{p-2} \as{u_x}^2 \dx + \left[f^\prime(0) + (\alpha C_2)^2 (p-1)\right] \int_\Omega u ^p \dx. 
\end{align}
Dropping the negative term, Gr{\"o}nwall's lemma implies that
\begin{align}\label{est:01.11}
    \sup_{t \in (0,T)} \norm{u (\cdot,t)}_{L^p (\Omega)} \leq e^{C_3 T} \norm{u_0}_{L^p(\Omega)} ,
\end{align}
where $C_3:= f^\prime(0) + (\alpha C_2)^2 (p-1)$. Hence, $\sup _{t \in (0,T)} \norm{u(\cdot,t)}_{L^p(\Omega)}$ is bounded for any $p > 1$ and any $T > 0$ fixed. Notice carefully that the exponent in estimate \eqref{est:1.11} depends critically on $p$, and so we cannot yet take the limit as $p \to +\infty$. 

Returning to \eqref{est:1.11}, we choose $p=2$ and integrate over $(0,T)$ to find that
\begin{align}
    \norm{u_x}_{L^{2,2}(Q_T)}^2 \leq C_3 \norm{u}_{L^{2,2}(Q_T)}^2 ,
\end{align}
and so $u_x$ is bounded in $L^{2,2}(Q_T)$ by estimate \eqref{est:01.11}.

We now apply $L^2$-theory for parabolic equations. Expanding the right-hand side of the equation for $u$, we have
$$
\alpha (u v_x )_x + f(u) = \alpha \left[ u_x v_x + u v_{xx} \right] + f(u) .
$$
Since $v_x \in L^\infty(Q_T)$ and $u_x \in L^{2}(Q_T)$, their product $v_x u_x \in L^{2}(Q_T)$. Similarly, since $u,\; v_{xx} \in L^{\infty,4} (Q_T)$, their product $u v_{xx} \in L^{2,2}(Q_T)$. Finally, from hypothesis \textbf{\textup{\ref{H2b}}} and the boundedness of $u$ in $L^\infty(0,T; L^p(\Omega))$ for any $p \geq 1$, it follows that $f(u) \in L^{2,2}(Q_T)$. By $L^2$-estimates for parabolic equations (see, e.g., \cite[Ch. 3]{Wu2006EllipticParabolic}), $u$ is bounded in $W^{2,1}_2 (Q_T)$. By the $t$-anisotropic Sobolev embedding (see, e.g., \cite[Theorem 1.4.1]{Wu2006EllipticParabolic}), we have that in fact $u \in C^{\sigma, \sigma/2} (\overline{Q}_T)$ for any $\sigma \in (0,1/2]$. In particular, $u \in L^\infty(Q_T)$. 

By hypothesis \textbf{\textup{\ref{H2b}}} and the boundedness of $u$ and $k$ over $Q_T$, it follows immediately that $\norm{k_t (\cdot,t)}_{L^p(\Omega)} = \norm{| g_1(u) - g_2(u) k|(\cdot,t)}_{L^p (\Omega)}$ is bounded, and estimate \eqref{est:ktLinf} follows.

Finally, we return to the equation for $u$ and multiply it by $u_{xx}$ to find
\begin{align}
    \frac{1}{2} \frac{{\rm d}}{{\rm d}t} \int_\Omega \as{u_x}^2 \dx + d \int_\Omega \as{u_{xx}}^2 \dx = \alpha \int_\Omega u_{xx} (u_x v_x + u v_{xx} ) \dx + \int_\Omega f(u) u_{xx} \dx .
\end{align}
Since $v_x$, $u \in L^\infty(Q_T)$, it is not difficult to see that by Cauchy's inequality, there holds
\begin{align}
    \frac{1}{2} \frac{{\rm d}}{{\rm d}t} \int_\Omega \as{u_x}^2 \dx + \frac{d}{2} \int_\Omega \as{u_{xx}}^2 \dx \leq & \, C_4 \int_\Omega \as{u_x}^2 \dx + C_5 (t) ,
\end{align}
where $C_4 = C_4 (\alpha, d, \norm{v_\xi}_{L^\infty(Q_T)})$ and 
$$
C_5(t) = \norm{u}_{L^\infty(Q_T)} ^2 \norm{v_{xx} (\cdot,t)}_{L^2(\Omega)}^2 + \norm{f(u(\cdot,t))}_{L^2(\Omega)}^2 \in L^1(0,T).
$$
Estimate \eqref{est:uW212} then follows via Gr{\"o}nwall's lemma.
\end{proof}

\subsection{Existence of a weak solution \& improved regularity}

We now prove the existence of a weak solution to system \eqref{1.18} and improve its regularity to that of a strong solution. Then, from Proposition \ref{prop:equivalentsystem}, any strong solution to problem \eqref{1.18} yields a strong solution to the nonlocal problem \eqref{1}. We then show that the solution is classical in the sense of Definition \ref{def:classicalsolution}. We begin with the following.
\begin{theorem}\label{thm:existweakstrong}
    Suppose hypotheses \textbf{\textup{\ref{H1b}}}-\textbf{\textup{\ref{H2b}}} hold, and that the initial data $(u_0, k_0)$ satisfies \eqref{Hinitialdata}. Then there exists a weak solution $(u,k,v)$ solving problem \eqref{1.18} in the sense of Definition \ref{def:weaksoln}. Moreover, the weak solution obtained is in fact the unique, global strong solution solving problem \eqref{1.18}. 
\end{theorem}
\begin{proof}[Proof of Theorem \ref{thm:existweakstrong}]
    Fix $T>0$. We first construct approximate solutions in a similar fashion to, e.g., \cite[Theorem 1.1]{liu2023biological} or \cite[Theorem 4.5]{Chazelle2017}, through iterates $(u_m, k_m, v_m)$ for $m \geq 2$  given by
\begin{equation}\label{1.18iterate}
   \begin{cases}
(u_m)_t = d (u_m)_{xx} +\alpha ( u_m (v_m)_{x})_{x} + f(u_m), & x\in (0,\pi), \ t>0, \\
(k_m)_t = g_1(u_{m-1}) - g_2(u_{m-1}) k_m, & x \in (0,\pi), \ t> 0, \\
0 = (v_m)_{xx}-\frac{1}{R^2}(v_m-k_m), & x \in (0,\pi), \ t> 0,
\end{cases} 
\end{equation}
where $(u_m(x,t), k_m(x,t), v_m(x,t))$ satisfy 
$$
\lim_{t\to0^+} (u_m(x,t), k_m(x,t), v_m(x,t)) = (u_0(x), k_0(x), G*k_0(x))
$$
uniformly in $[0,\pi]$.\\ 

\noindent\textbf{Step 1:} We claim that the sequence $\{ (u_m, k_m, v_m) \}_{m\geq2}$ generated by \ref{1.18iterate} is well defined. Indeed, given a smooth initial iterate $u_1(x,t)$ satisfying a homogeneous Neumann boundary condition (e.g., take $u_1$ to be the solution to the heat equation with $u_1(x,0) = u_0(x)$), $k_2(x,t)$ is well-defined by solving the ODE for each $x$ fixed. So, $k_2(x,t)$ is smooth, positive in $Q_T$, and satisfies the homogeneous Neumann boundary condition. By the theorem of Lax-Milgram, there exists a smooth solution $v_2(x,t)$ solving the elliptic problem for each $t \in (0,T)$ fixed. Moreover, the solution is unique subject to the scaling $\int_0 ^\pi v_2(x,t) \dx = \int_0 ^\pi k_2(x,t) \dx$ for each $t\in(0,T)$. By Schauder theory for parabolic equations, the smoothness of $v_2$ ensures that there exists a unique smooth solution $u_2(x,t)$ satisfying the boundary condition and $\lim_{t \to 0^+} u_2(x,t) = u_0 (x)$. One may then proceed inductively, using the smoothness of the iterate $u_{m-1}(x,t)$ to obtain the existence of a unique solution $(u_m, k_m,v_m)$ for any $m \geq 3$. \\

\noindent\textbf{Step 2:} We now obtain uniform bounds on this sequence. By Lemma \ref{ODEbounds1} we have
\begin{align*}
    \sup_{t \in (0,T)} \norm{k_m (\cdot,t)}_{L^\infty(\Omega)} &\leq M + \norm{k_0}_{L^\infty(\Omega)},
\end{align*}
where $M$ depends on $g_1$, $g_2$ but remains independent of $m$. In particular, $\{k_m\}_{m \geq 2}$ is bounded in $L^\infty (Q_T)$. By Lemma \ref{lem:aprioriestimates1}, we then have the following uniform estimates on the sequences $\{u_m\}_{m\geq2}$, $\{ k_m\}_{m\geq2}$ and $\{ v_m\}_{m\geq2}$:
\begin{enumerate}[(A)]
    \item $\{ u_m \}_{m\geq2}$ is bounded in $L^\infty(Q_T)\cap W^{2,1}_2 (Q_T)$;
    \item $\{ k_m \}_{m\geq2}$ is bounded in $H^1(0,T; L^\infty(\Omega))$;
    \item $\{ v_m \}_{m \geq 2}$ is bounded in $L^\infty(0,T; W^{2,p} (\Omega))$ for any $p \geq 1$.
\end{enumerate}

\noindent\textbf{Step 3:} We now extract a convergent subsequence and show that it is a weak solution to problem \eqref{1.18} in the sense of Definition \ref{def:weaksoln}. 

First, (A) implies the existence of a limit candidate, denoted by $u$, and a subsequence (still denoted by $m$) such that $u_m \to u$, $(u_m)_x \to u_x$ strongly in $L^2(Q_T)$ as $m \to +\infty$. Moreover, there exists a limit $u_t$ such that (up to subsequence) $(u_m)_t \to u_t$ weakly in $L^2(Q_T)$ as $m \to +\infty$. By Hypothesis \textbf{\ref{H2a}} (local Lipschitz continuity is sufficient) and the uniform boundedness of $\{u_m\}_{m\geq2}$ there holds $f(u_m) \to f(u)$ strongly in $L^2(Q_T)$ as $m \to +\infty$. 

Then, (C) implies that for each $t \in (0,T)$ fixed there exists a limit $v_x$ such that $(v_m)_x \to v_x$ strongly in $L^2(\Omega)$, and so too in $L^2(Q_T)$. Since $u_m \to u$ pointwise by the uniform H{\"o}lder continuity of the iterates, we conclude that $u_m (v_m)_x \to u v_x$ strongly in $L^2(Q_T)$. 

Therefore, given any test function $\phi \in L^2(0,T; H^1(\Omega))$ there holds in the limit as $m \to +\infty$:
\begin{align}
    \iint_{Q_T} u_t \phi \dx {\rm d}t + \iint_{Q_T} \phi_x u_x \dx {\rm d}t = & \, \alpha \iint_{Q_T} \phi_x u v_x \dx {\rm d}t + \iint_{Q_T} \phi f(u) \dx {\rm d}t. 
\end{align}
Since $\{ u_m\}_{m\geq2}$ is bounded in $W^{2,1}_2 (Q_T)$, by the Sobolev embedding we have that $\{ u_m \}_{m \geq2}$ is bounded in $C^{\sigma, \sigma/2} (\overline{Q}_T)$ for some $\sigma \in (0,1)$, and so in fact $u_m \to u$ pointwise as $m\to +\infty$. Solving directly for $k_m$ we find that there holds $k_m \to k$ pointwise as $m \to +\infty$ as well. Consequently, $w(u_m, k_m) \to w(u,k)$ pointwise as $m \to +\infty$, and so converges strongly in $L^2(Q_T)$. By estimate (B), there exists a subsequence such that $(k_m)_t \to k_t$ weakly in $L^2(Q_T)$ as $m \to +\infty$. Hence, in the limit, there holds
\begin{align}
    \iint_{Q_T} k_t \phi \dx {\rm d}t = \iint_{Q_T} \phi w(u,k) \dx {\rm d}t.
\end{align}
Finally, in the limit as $m \to +\infty$ there holds for any test function $\phi$
\begin{align}
    \iint_{Q_T} \phi_x v_x \dx {\rm d}t = \frac{1}{R^2} \iint_{Q_T} \phi ( k - v) \dx {\rm d}t .
\end{align}

Consequently, $(u,k,v)$ is a global weak solution solving problem \eqref{1.18} in the sense of Definition \ref{def:weaksoln}. The since $u,k \in H^1(0,T; L^2(\Omega))$, we have that in fact $u,k \in C(0,T; L^2(0,T))$ and the initial data is satisfied in the sense of $L^2(\Omega)$. 

By the H{\"o}lder continuity of $u$, we immediately have that $\lim_{t \to 0^+} u(\cdot,t) = u_0(\cdot)$ in $C(\overline{\Omega})$. Then, since $ k \in H^1 (0,T; L^\infty(\Omega))$, we have that $k \in C( 0,T; L^\infty (\Omega))$ and so there holds $\lim_{t \to 0^+} k(\cdot,t) = k_0 (\cdot)$ in $C(\overline{\Omega})$. From the estimates of Lemma \ref{lem:aprioriestimates1}, we have that in fact the solution is a global strong solution.
\end{proof}

Finally, we conclude with a brief proof of the main Theorem \ref{thm:wellposed}.

\begin{proof}[Proof of Theorem \ref{thm:wellposed}]
By Theorem \ref{thm:existweakstrong}, there exists a unique, global strong solution $(U,K,V)$ solving problem \eqref{1.18} in $(0,\pi)$. By Proposition \ref{prop:equivalentsystem}, we may reflect $(U,K)$ across $x=0$ and extend the result periodically over $\mathbb{R}$ to obtain a global, even, $2\pi$-periodic strong solution $(u,k)$ solving problem \eqref{1}. By Lemma \ref{lem:uniqueness}, the solution is unique. It remains to show that the solution is classical.

To show that the solution is classical, we show that 
$(u (G*k)_x)_x + f(u)$ is H{\"o}lder continuous and apply Schauder theory for parabolic equations. To this end, we first note that we already have $u,k \in C^{\sigma, \sigma/2} (\overline{Q}_T)$ for some $\sigma \in (0,1)$, and so $\as{f(u)} \in L^\infty(Q_T)$.

Next, using the property that $(G * k)_{xx} = (1/R^2) ( G*k - k)$, it follows from the boundedness of $k$ in $L^p(Q_T)$ that $(G * k)_{xx} \in L^p(Q_T)$ for any $p \geq 1$ fixed. The boundedness of $u$ implies that the product $u (G*k)_{xx} \in L^p(Q_T)$ for any $p \geq 1$.

Finally, we improve the integrability of $u_x$ as follows. Since $u_x \in L^\infty(0,T; L^2(\Omega))$ with $u_{xx} \in L^2(Q_T)$, we have that $u_x \in V_2(Q_T)$, where $V_2(Q_T)$ is the Banach space comprised of functions belonging to $L^\infty(0,T; L^2(\Omega)) \cap W^{1,0}_2 (Q_T)$ with finite norm $\norm{z}_{V_2(Q_T)} := \norm{z}_{L^\infty(0,T; L^2(\Omega))} + \norm{z_x}_{L^2(Q_T)}$ (see, e.g., \cite[Ch. 1]{Wu2006EllipticParabolic}). By the embedding theorem \cite[Theorem 1.4.2]{Wu2006EllipticParabolic}, we have that there holds $\norm{u_x}_{L^{10/3} (Q_T)} \leq C \norm{u_x}_{V_2(Q_T)}$ for some uniform constant $C>0$. Since $k \in L^\infty(Q_T)$, we have $(G*k)_x \in L^\infty (Q_T)$ by direct computation and Young's convolution inequality, and so we have that the product $u_x (G*k)_x \in L^{10/3}(Q_T)$. 

Together, we have shown that $(u (G*k)_x)_x + f(u) \in L^{10/3} (Q_T)$, and so by $L^p$-estimates for parabolic equations we conclude that in fact $u \in W^{2,1}_{10/3} (Q_T)$. Using again the $t$-anisotropic Sobolev embedding we find that $u \in C^{1+\sigma, (1+\sigma)/2} (\overline{Q}_T)$ for $\sigma = 1/10$. In particular, we have shown that $u$ is once differentiable in space. 

Solving the differential equation for $k$, the regularity from Hypothesis \textbf{\ref{H2a}} ensures that $k$ is also once differentiable in space. Hence, the products $u_x G* k_x$ and $u (G*k)_{xx} = u ( G*k - k)/ R^2$ are H{\"o}lder continuous. Since $f$ is differentiable, it is obviously H{\"o}lder continuous. By Schauder estimates for parabolic equations (see, e.g., \cite[Ch. 7]{Wu2006EllipticParabolic}) and the regularity of the initial data, we conclude that $u \in C^{2+\sigma, 1 + \sigma/2} (\overline{Q}_T)$ for some $\sigma \in (0,1)$. Direct calculation and the regularity for $g_i(\cdot)$ from hypothesis \textbf{\ref{H2b}} ensures that $k \in C^{2+\sigma, 1 + \sigma/2} (\overline{Q}_T)$, and the solution is classical. 
\end{proof}

\section{Bifurcation Analysis}\label{sec:biftheory}

This section is dedicated to proving Theorems \ref{thm:1.2}-\ref{thm:bif3}. 
We begin with a local stability analysis of $(u_*,k_*,v_*)$.

\subsection{Local stability of the constant steady state}

From \eqref{1.11}, the characteristic equation of \eqref{1.11} is
\begin{equation}\label{1.12}
\begin{aligned}
    H(\lambda, \alpha, n^2):=\lambda^2 - T(n^2)\lambda + D(\alpha,n^2)=0,
    \end{aligned}
\end{equation}
where $T(n^2)$ and $D(\alpha,n^2)$ be defined in \eqref{TD}.

It is well known that $\lambda<0$ always holds when $T(n^2)<0$ and $D(n^2, \alpha)>0$, which yields that the positive steady state solution is always stable. However, the steady state solution becomes unstable by having $\lambda=0$ as an eigenvalue through the steady-state bifurcation ($D(n^2,\alpha)=0$) or purely imaginary eigenvalues $\lambda=\pm {\rm i}\omega (\omega>0)$ through Hopf bifurcation ($T(n^2)=0$). Noting that $T(n^2)<0$ for any $n\in \mathbb{R}$ by assumptions \textbf{\textup{\ref{H1b}}} and \textbf{\textup{\ref{H2a}}}, \eqref{1.18} does not exhibit Hopf bifurcation from the constant steady state solution $(u_*,k_*,v_*)$. Moreover, $D(n^2,\alpha)>0$ when $w_{u*}=0$. Therefore, if $w_{u*}\neq 0$, using $\alpha$ as the bifurcation parameter, we can obtain the bifurcation point given by \eqref{1.13}.

Therefore, when condition $\alpha=\alpha_n(R)$ is satisfied, $0$ is an eigenvalue of the characteristic equation \eqref{1.12}, and the eigenvalue is a simple one. We now check the transversality condition. Differentiating \eqref{1.12} with respect to $\alpha$ letting $\lambda=0$ gives 
\begin{equation*}
    \dfrac{\partial \lambda}{\partial \alpha}=-\dfrac{u_* w_{u*} \dfrac{n^2}{1+n^2 R^2}}{d n^2 -f_{u*} - w_{k*}}\neq 0 \ \text{for}\ n^2\neq 0.
\end{equation*}
Hence, the transversality condition holds. We have the following concavity result for $\alpha_n (R)$ with respect to $n$ for a given, fixed $R>0$.

\begin{lemma}\label{lem:1}
Fix $R>0$ and let $\alpha_n(R)$ be as defined in \eqref{1.13}. Then, the function $n^2 \mapsto \alpha_n (R)$ is concave down (up) whenever $w_{u*}>0$ ($w_{u*}<0$). Moreover, $\alpha_n (R)$ attains its maximum (minimum) at $n^* = (- f_{u*} / d R^2 )^{1/4}$ whenever $w_{u*}>0$ ($w_{u*}<0$).
\end{lemma}
\begin{proof}
  Set $z=n^2$ in \eqref{1.13} so that we have 
  \begin{equation*}
    \alpha_z(R) = \dfrac{(d z - f_{u*}) w_{k*} (1+z R^2)}{u_* w_{u*} z}, 
\end{equation*}
Direct calculation yields
\begin{equation*}
    \dfrac{\partial\alpha_z(R)}{\partial z} = \dfrac{w_{k*}(d z^2 R^2 + f_{u*})}{u_* w_{u*} z^2},
\end{equation*}
and 
\begin{equation*}
    \dfrac{\partial^2\alpha_z(R)}{\partial z^2} = \dfrac{-2w_{k*}f_{u*}}{u_* w_{u*} z^3}=\begin{cases}
        <0,\ \text{if}\ w_{u*}>0,\\
        >0,\ \text{if}\ w_{u*}<0
    \end{cases}
    \ \text{by \textbf{\textup{\ref{H1a}}} and \textbf{\textup{\ref{H2a}}}}.
\end{equation*}
\end{proof}

We now prove Theorem \ref{thm:1.2}.
\begin{proof}[Proof of Theorem \ref{thm:1.2}]
The characteristic equation of \eqref{1.18} is the same as \eqref{1.12}, \eqref{TD} and \eqref{1.13} and yields (\romannumeral1) and (\romannumeral2). From Lemma \ref{lem:1}, it follows that $\alpha_*$ or $\alpha^*$ exist, and hence (\romannumeral3) and (\romannumeral4) hold.
\end{proof}

\subsection{Proof of Theorem \ref{thm:bif2}}

We now prove Theorem \ref{thm:bif2}.

\begin{proof}[Proof of Theorem \ref{thm:bif2}]

Recall that we fix $d,\mu,\beta>0$ and treat $\alpha$ as the bifurcation parameter. We define a nonlinear mapping $F: \mathbb{R} \times X \times Y\times X \rightarrow Y^3$ by
\begin{equation}
F(\alpha,u,k,v)=\left(\begin{array}{c}
d u_{xx} + \alpha ( u v_{x})_{x} + f(u)
\\ 
g_1(u) - g_2(u) k\\
v_{xx}-\frac{1}{R^2}(v-k)
\end{array}\right).
\end{equation}
The Fr\'{e}chet derivative of $F$ with respect to $(u,k,v)$ at $(u_*,k_*,v_*)$ is
\begin{equation}
    F_{(u,k,v)}(\alpha_n,u_*,k_*,v_*)\left(\begin{array}{c}
\phi
\\ 
\varphi\\
\psi
\end{array}\right)=
\left(\begin{array}{c}
d \phi_{xx} + \alpha_n u_* \psi_{xx}  + f_{u*}\phi
\\ 
w_{u*} \phi + w_{k*} \varphi
\\
\psi_{xx}-\frac{1}{R^2}(\psi-\varphi)
\end{array}\right).
\end{equation}
We now check the hypotheses of the theorem \cite[Theorem 1.17]{Rabinowitz1971}. Clearly, $F(\alpha,u_*,k_*,v_*)=0$ for any $\alpha\in\mathbb{R}$. Let $F^*_{(u,k,v)}(\alpha_n,u,k,v)$ be the adjoint operator of $F_{(u,k,v)}(\alpha_n,u,k,v)$, we have 
\begin{equation}
    F^*_{(u,k,v)}(\alpha_n,u_*,k_*,v_*)\left(\begin{array}{c}
\phi
\\ 
\varphi\\
\psi
\end{array}\right)=
\left(\begin{array}{c}
d \phi_{xx} + f_{u*}\phi  + w_{u*}\varphi
\\ 
w_{k*} \varphi + \frac{1}{R^2} \psi
\\
 \alpha_n u_*\phi_{xx} + \psi_{xx} -\frac{1}{R^2}\psi
\end{array}\right).
\end{equation}
It is easily seen that 
\begin{equation}
\begin{aligned}
     \mathcal{N}(F_{(u,k,v)}(\alpha_n,u_*,k_*,v_*)) = \mathrm{span}\{q_n\},\ \text{where}\ q_n=(1,M_1,M_2)\cos(nx),\\ \mathcal{N}(F^*_{(u,k,v)}(\alpha_n,u_*,k_*,v_*)) = \mathrm{span}\{q_n^*\},\ \text{where}\ q_n^*=(1,M^*_1,M^*_2)\cos(nx),
\end{aligned}
\end{equation}
and 
\begin{equation}\label{pp}
\begin{split}
    &\mathcal{R}(F_{(u,k,v)}(\alpha_n,u_*,k_*,v_*))\\
    =& \{(h_1,h_2,h_3)\in Y^3: \int_0^{\pi} (h_1 + M^*_1 h_2 + M^*_2 h_3) \cos(nx) \dx =0\},
\end{split}  
\end{equation}
where $M_1, M_2$ are defined in \eqref{1.24}, and $M^*_1, M^*_2$ are as follows:
\begin{equation}\label{M1*M2*}
    M^*_1=\dfrac{d n^2-f_{u*}}{w_{u*}},\ \ M^*_2=-\dfrac{(d n^2-f_{u*})w_{k*}R^2}{w_{u*}}=-R^2w_{k*} M^*_1.
\end{equation}
Thus, 
\begin{equation*}
    {\rm dim}(\mathcal{N}(F_{(u,k,v)}(\alpha_n,u_*,k_*,v_*)))={\rm codim}(\mathcal{R}(F_{(u,k,v)}(\alpha_n,u_*,k_*,v_*)))=1.
\end{equation*}
We next show that $F_{\alpha(u,k,v)}(\alpha_n,u_*,k_*,v_*))[q_n] \notin \mathcal{R}(F_{(u,k,v)}(\alpha_n,u_*,k_*,v_*))$.
It follows that 
\begin{equation}
    F_{\alpha(u,k,v)}(\alpha_n,u_*,k_*,v_*))[q_n] = (-u_* n^2 M_2 \cos(nx), 0 ,0)^{\rm T},
\end{equation}
we have 
\begin{equation}
    \int_0^\pi (-u_* n^2 M_2 \cos(nx)+0+0)\cos(nx) \dx \neq 0,
\end{equation}
so $F_{\alpha(u,k,v)}(\alpha_n,u_*,k_*,v_*))[q_n] \notin \mathcal{R}(F_{(u,k,v)}(\alpha_n,u_*,k_*,v_*))$ from \eqref{pp}.
Now we obtain the bifurcating solutions on $\Gamma_n$ in \eqref{1.22} by applying the Theorem \cite[Theorem 1.17]{Rabinowitz1971} of Crandall and Rabinowitz. The conclusion of $\alpha_n'(0)=0$ follows from the calculation in Appendix \ref{sec:AppdExpression}.
\end{proof}

\subsection{Proof of Theorem \ref{thm:bif3}}

Finally, we conclude with a proof of Theorem \ref{thm:bif3}.
\begin{proof}[Proof of Theorem \ref{thm:bif3}]
    From the Theorem 1.16 in \cite{Crandall1973}, there exist continuously differentiable function $r: (\alpha_n-\epsilon, \alpha_n+\epsilon)\rightarrow \mathbb{R}$ and $(\phi,\varphi,\psi): (\alpha_n-\epsilon, \alpha_n+\epsilon)\rightarrow X\times Y\times X$ such that 
    \begin{equation}
F_{(u,k,v)}(\alpha,u_*,k_*,v_*)\left[\phi(\alpha),\varphi(\alpha),\psi(\alpha)\right]^{\rm T} = r(\alpha) K \left[\phi(\alpha),\varphi(\alpha),\psi(\alpha)\right]^{\rm T},
	\end{equation}
where $K: X\times Y\times X\to Y^3$ is the inclusion map, $r(\alpha_n)=0$, $(\phi(\alpha_n),\varphi(\alpha_n),\psi(\alpha_n))^{\rm T}=q_n$. In fact, $r(\alpha)$ is the eigenvalue of the corresponding linearization operator for the constant solution $(u_*,k_*,v_*)$ and $(\phi(\alpha_n),\varphi(\alpha_n),\psi(\alpha_n))$ is the corresponding eigenfunction.
According to Eq. \eqref{1.12}, 
\begin{equation*}
     r'(\alpha_{n})=\dfrac{-u_* w_{u*}\frac{n^2}{1+n^2R^2}}{dn^2-f_{u*}-w_{k*}}.
\end{equation*}
By \cite[Theorem 1.16]{Crandall1973}, ${\rm Sign}(-s\alpha_n'(s)r'(\alpha_n))={\rm Sign}(\mu(s))$, where $\mu(s)$ is the eigenvalue of the corresponding linearization operator for the bifurcating solution at $\alpha=\alpha(s)$. Here we only prove (\romannumeral1), (\romannumeral2) can be proved similarly. 
If $w_{u*}>0$ and $\alpha_*=\alpha_M$ for $M\in {\mathbb N}$, then $r'(\alpha_{M})<0$. Moreover,  
If $\alpha_M''(0)<0$, then $\alpha_M'(s)>0$  for $s\in(-\delta,0)$, and $\alpha_M'(s)<0$  for $s\in(0,\delta)$. Thus $\mu_M(s)<0$ for $0<|s|<\delta$. Since all other eigenvalues of linearized equation at $(\alpha_M(s),u_M(s,\cdot),k_M(s,\cdot),v_M(s,\cdot))$ are negative, then $(\alpha_M(s),u_M(s,\cdot),k_M(s,\cdot),v_M(s,\cdot))$ is locally asymptotically stable. Similarly if $\alpha_M''(0)>0$, then $\mu_M(s)>0$ for $0<|s|<\delta$ hence $(\alpha_M(s),u_M(s,\cdot),k_M(s,\cdot),v_M(s,\cdot))$ is unstable. 
\end{proof}

\section{Applications}\label{sec:apps}

This section further explores our theoretical results through additional analysis and numerical simulation. While much of our main insights can be viewed simply by reading Section \ref{sec:interpretationofmainresults}, here we explore in more detail the influence of the encoding processes $g_1$ and $g_2$.

In our following simulations, we always fix 
$$
k_0 = v_0 = k_*,
$$ 
and choose the initial data $u_0$ as a random perturbation of the homogeneous state $u_*$. The perturbation is chosen from a uniformly distributed random variable in the interval $(-0.01,0.01)$ and normalized with $0$ mean. We fix the domain $\Omega=(-\pi,\pi)$ and $f(u)=u(1-u)$ so that there always holds $u_* = 1$. We also fix 
$$
g_2(u) := \mu + \beta u,
$$ 
for constants $\mu, \beta >0$. Due to the equivalence between system \eqref{1} and \eqref{1.18}, we solve the local problem \eqref{1.18} on the interval $(0,\pi)$ using MATLAB's built-in ``pdepe" function before reverting to the original solution on $\Omega = (-\pi,\pi)$. We run our simulations until the approximate time derivative $u_t$ has maximum less than $10^{-8}$, i.e., until $\norm{u_t}_\infty \leq 10^{-8}$. 

\subsection{Description of numerical bifurcation scheme} 

We discretize a window $(\alpha^* - \delta_0, \alpha^* +\delta_0)$ with approxmiately $40-60$ points about the critical threshold $\alpha^*$ obtained through our linear stability analysis. Starting from $\alpha = \alpha^*-\delta_0$ (assuming $\alpha^*>0$; the signs are reversed when $\alpha^*<0$), we run each simulation with the perturbed initial data to verify the local stability analysis. To capture possible subcritical behavior, we run a secondary sweep starting from $\alpha = \alpha^* + \delta_0$, using the terminal state of the $i^{\textup{th}}$ iteration as the initial guess for the $(i+1)^{\textup{th}}$ iteration. In those cases where we predict a subcritical bifurcation branch to occur, we can then detect this stable portion of the curve lying above the unstable portion that is not detectable by our time-dependent solver.

We are then interested in the steady state profiles depending on the positive encoding process $g_1(\cdot)$, the perceptual radius $R$, and the aggregation strength $\alpha$. We first make the following general observation: the relative rates of change of the encoding processes $g_1$ and $g_2$ determine whether the population distribution $u^*$ and the spatial map $k^*$ are \textit{in} or \textit{out} of phase with each other. This is encoded in the quantity $w_{u^*}$, whose sign depends on the sign of \eqref{q:wustar}. Hence, from \eqref{1.24}, we conclude that
\begin{align}\label{eq:solutionsinphase}
    \text{Solutions are in phase} \iff \frac{g_1^\prime (u^*)}{g_1(u^*)} > \frac{g_2^\prime (u^*)}{g_2(u^*)} = \frac{\beta}{\mu + \beta u^*}.
\end{align}
In other words, despite the kernel $G$ being held fixed so that $\alpha>0$ corresponds to repulsion, $\alpha<0$ corresponds to attraction, the spatial map $k$ may change from an \textit{attractive spatial map} to a \textit{repulsive spatial map} whenever the sign relation above changes: whether solutions are in or out of phase depends on whether the relative rate of change of the positive encoding process $g_1$ is larger or smaller than the negative encoding process $g_2$. Interestingly, this relation holds independent of the detection radius $R\geq0$. 

Therefore, we explore two key examples for which these signs change. Using Theorems \ref{thm:1.2}-\ref{thm:bif2}, we locate the first critical value $\alpha^*$ and identify the wavenumber at which patterned states are expected to emerge. Using Theorem \ref{thm:bif3}, we then compute $\alpha^{\prime \prime}(0)$ to determine the direction and stability of the bifurcating branch.\\

\subsection{Three exemplary cases}\label{sec:examples}

In what follows, we consider three distinct cases depending on the form of $g_1(u)$. Recall that we have fixed $g_2$ to be linear in $u$. Then, the first example below has approximately linear growth of $g_1(\cdot)$ while the second example has sublinear growth of $g_1(\cdot)$. The third example, while pathological, highlights that, in terms of pattern formation, the growth rate of the functions $g_i(\cdot)$ do not matter as much as the relative growth rates near $u^*$, as emphasized in Section \ref{sec:interpretationofmainresults}.

\begin{examp}\label{ex1}
    We set $g_1(u)= \rho u^2 / (1+u)$. The constant state for the spatial map is $k_*=\rho/(2(\mu+\beta))$.
The remaining parameters are fixed as follows:
\begin{align}\label{parameter}
    d=1,\ \rho=1,\ \mu=0.15,\ \beta=0.5.
\end{align}
and so we compute
\begin{align*}
 u_*=1,\  k_*=v_*\approx0.7692,\ w_{u*} \approx 0.3654>0,\ w_{k*}\approx -0.65.
\end{align*}
We first observe that for \textit{any} $\mu,\beta > 0$ (not just those fixed above) there holds
$$
\frac{g_1^\prime (u^*)}{g_1(u^*)} - \frac{g_2^\prime (u^*)}{g_2(u^*)} = \frac{2 \mu + u^*(2 \beta + \mu)}{u^* (1+u^*)(\mu + \beta^*)} > 0,
$$
and so by relation \eqref{eq:solutionsinphase}, the population distribution and spatial map will always be in phase. Consequently, we observe that by \eqref{1.13} there holds $\alpha_n<0$ for any $n \in \mathbb{N}$, and so a patterned state will only emerge due to attractive forces. 

In Figure \ref{fig:1}, we plot the critical value $| \alpha_*|$ as a function of the perceptual radius $R$ as a black line, and the critical wavenumber $n$ is given by the grey dashed line. The shaded area depicts the region of (local) stability for the homogeneous state in the $(R,\alpha)$ plane. We choose different values of $R$ so that the corresponding value $n$ to the critical threshold $\alpha_*$ is either $1$, $2$, or $3$.

In particular, we choose $R=0.12, 0.3, 2$, with the intersection between $R$ and $|\alpha_*|$ highlighted by red dots, labeled as the following points in Figure \ref{fig:1}:
\begin{align*}
    & H:= (0.12,-2.2328);\ n=3,\ \alpha_n''(0) \approx -0.9055<0;\\
    & Q:= (0.3,-3.0242);\ n=2,\ \alpha_n''(0)\approx -0.1810<0;\\
    & P:= (2,-17.7895);\ n=1,\  \alpha_n''(0)\approx 16.1136>0.
\end{align*}
In each case, we indicate the value $n$ at which $\alpha_*$ is achieved and compute $\alpha_n''(0)$.

In Figures \ref{fig:2}, \ref{fig:3}, and \ref{fig:4}, we fix each respective value of $R$ corresponding to points $H$, $Q$ and $P$, and depict the bifurcation curve locally around the corresponding critical threshold $\alpha^*(R)$. From the sign of $\alpha_n ''(0)$, we can identify whether the bifurcation direction is forward or backward. 

At points $H$ (Figure \ref{fig:2}) and $Q$ (Figure \ref{fig:3}), we find that $\alpha_{n}'' (0) < 0$ and a stable backward pitchfork bifurcation occurs by Theorem \ref{thm:bif3} (note that it is backwards due to $\alpha^* < 0$). Moreover, these occur at the expected wavenumber $n=3$ and $n=2$.

Conversely, at point $P$ we find $\alpha_n^{\prime\prime} (0)>0$, and the bifurcation is forward and unstable by Theorem \ref{thm:bif3}. Moreover, we expect this low-amplitude state to occur with frequency $n=1$; however, since it is now an unstable state, this is not what is detected by the time-dependent solver used here. Instead, we observe a high-amplitude state with a single aggregate. In Figure \ref{fig:4}, we observe precisely this subcritical behavior and stability properties described analytically.

\begin{figure}
    \centering
    \includegraphics[trim={4cm 4cm 5cm 4cm},clip,width=\linewidth]{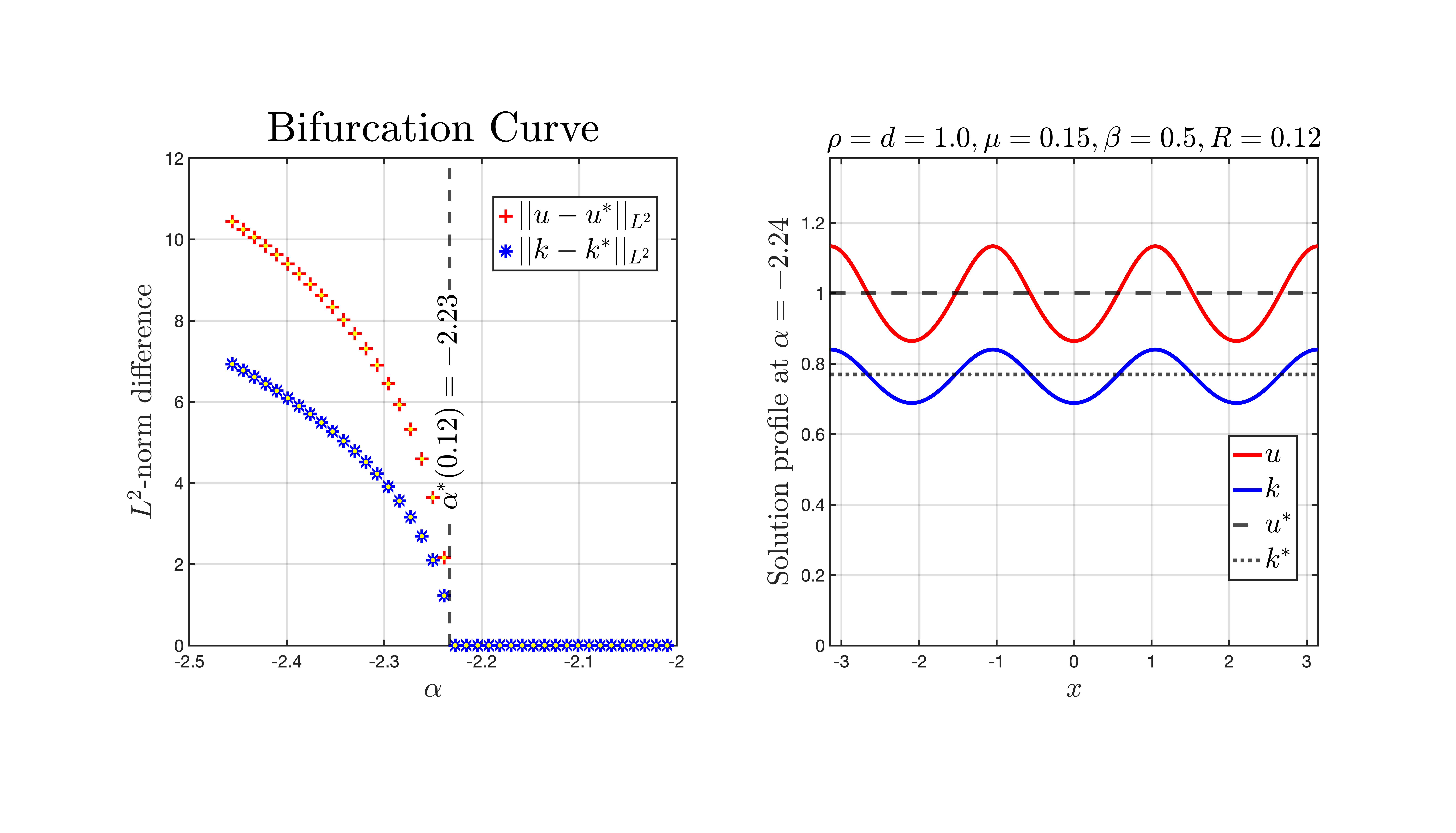}
    \caption{A bifurcation diagram near the critical threshold $\alpha^*$ when $R=0.12$ for \textbf{Example} \ref{ex1} of Section \ref{sec:apps} (left panel) and the solution profile at steady state just beyond the critical threshold (right panel).}
    \label{fig:2}
\end{figure}
\end{examp} 

\begin{examp}\label{ex2}
We now set $g_1(u) = \rho u/(1+u)$ with all parameter choices consistent with \eqref{parameter}, and all with all other functions the same as in \textbf{Example} \ref{ex1}. We then compute
\begin{align*}
    u_*=1,\ k_*=v_*\approx 0.7692,\ w_{u*}\approx -0.1346<0,\ w_{k*}\approx -0.65.
\end{align*}
In fact, $k_*$ and $w_{k*}$ are the same as in \textbf{Example} \ref{ex1}; the key difference is the sign of $w_{u*}$, which depends precisely on the quantity 
$$
\mu - \beta u_*^2 = \mu - \beta.
$$
From our parameter choices, $\sign (w_{u*}) = \sign(\mu - \beta) < 0$, and therefore, the sign of the critical threshold is reversed so that $\alpha ^* > 0$. Moreover, we now expect patterned states, whenever they occur, to be out of phase as predicted by Theorem \ref{thm:bif2}, since we now find $M_1<0$. Similar to the curve depicted in Figure \ref{fig:1} for \textbf{Example} \ref{ex1}, we can identify the critical aggregation strengths for varying $R$, along with the corresponding critical wavenumber for \textbf{Example} \ref{ex2}. Since the curve is similar, we do not display it here. At all points $H$, $Q$ and $P$, we find that $\alpha^{\prime \prime}_n (0) > 0$, and so all emergent branches are expected to be supercritical and stable.

In Figure \ref{fig:6}, we display the same bifurcation figures as for \textbf{Example} \ref{ex1}, noting the consistent supercritical behavior, stability of the emergent branches, and the phase properties expected for these cases. Interestingly, depending on the relative sizes of $\mu$ and $\beta$, the population may be in or out of phase with the spatial map. This fundamentally differs from \textbf{Example} \ref{ex1}, where the population and spatial map are always in phase regardless of the chosen parameter values.
\end{examp}

\begin{figure}
    \centering
    \includegraphics[trim={4cm 4cm 5cm 4cm},clip,width=\linewidth]{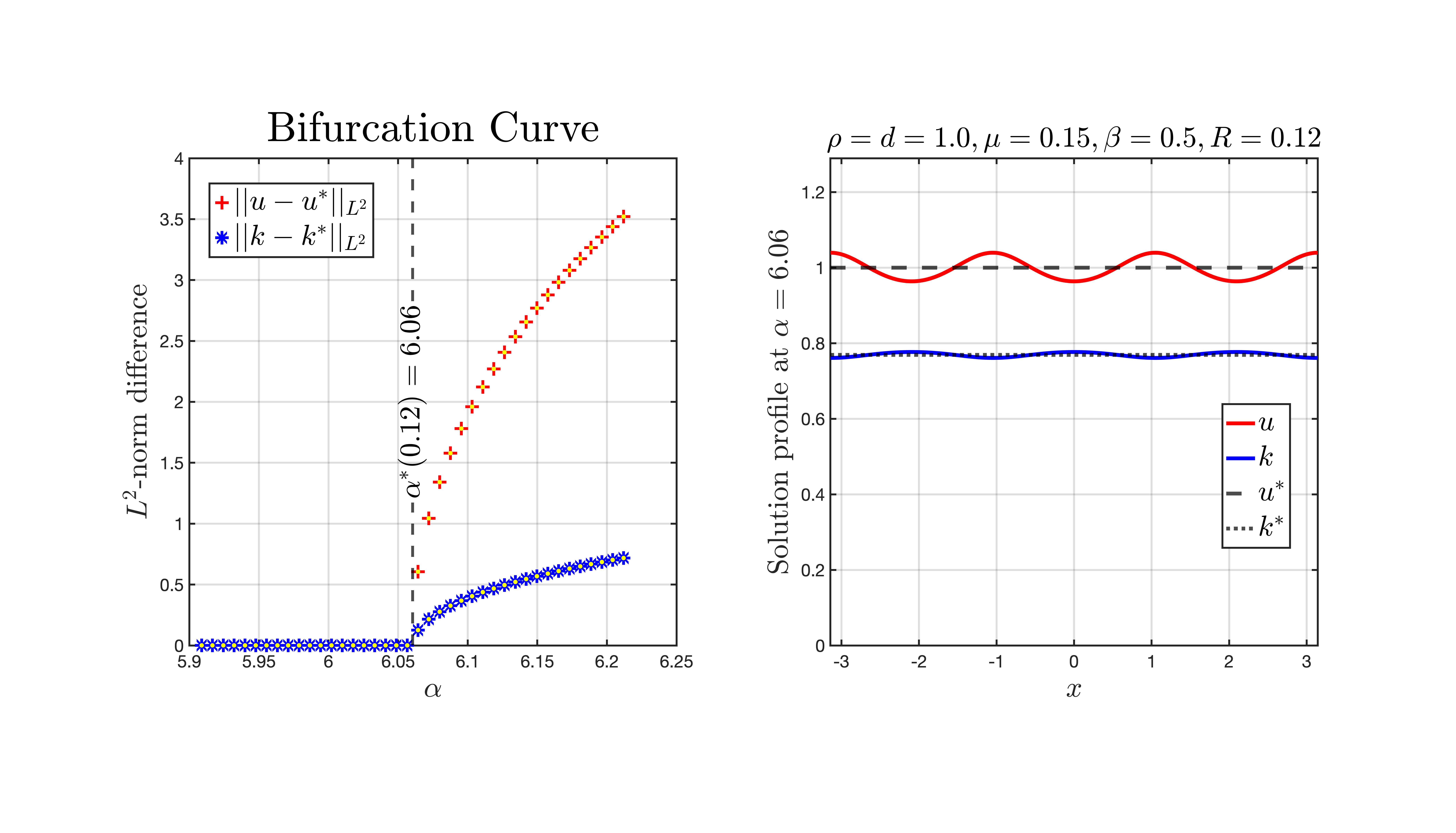}
    \caption{A bifurcation diagram near the critical threshold $\alpha^*$ when $R=0.12$ for \textbf{Example} \ref{ex2} of Section \ref{sec:apps} (left panel) and the solution profile at steady state just beyond the critical threshold (right panel).}
    \label{fig:5}
\end{figure}

\begin{figure}
    \centering
    \includegraphics[trim={4cm 4cm 5cm 4cm},clip,width=\linewidth]{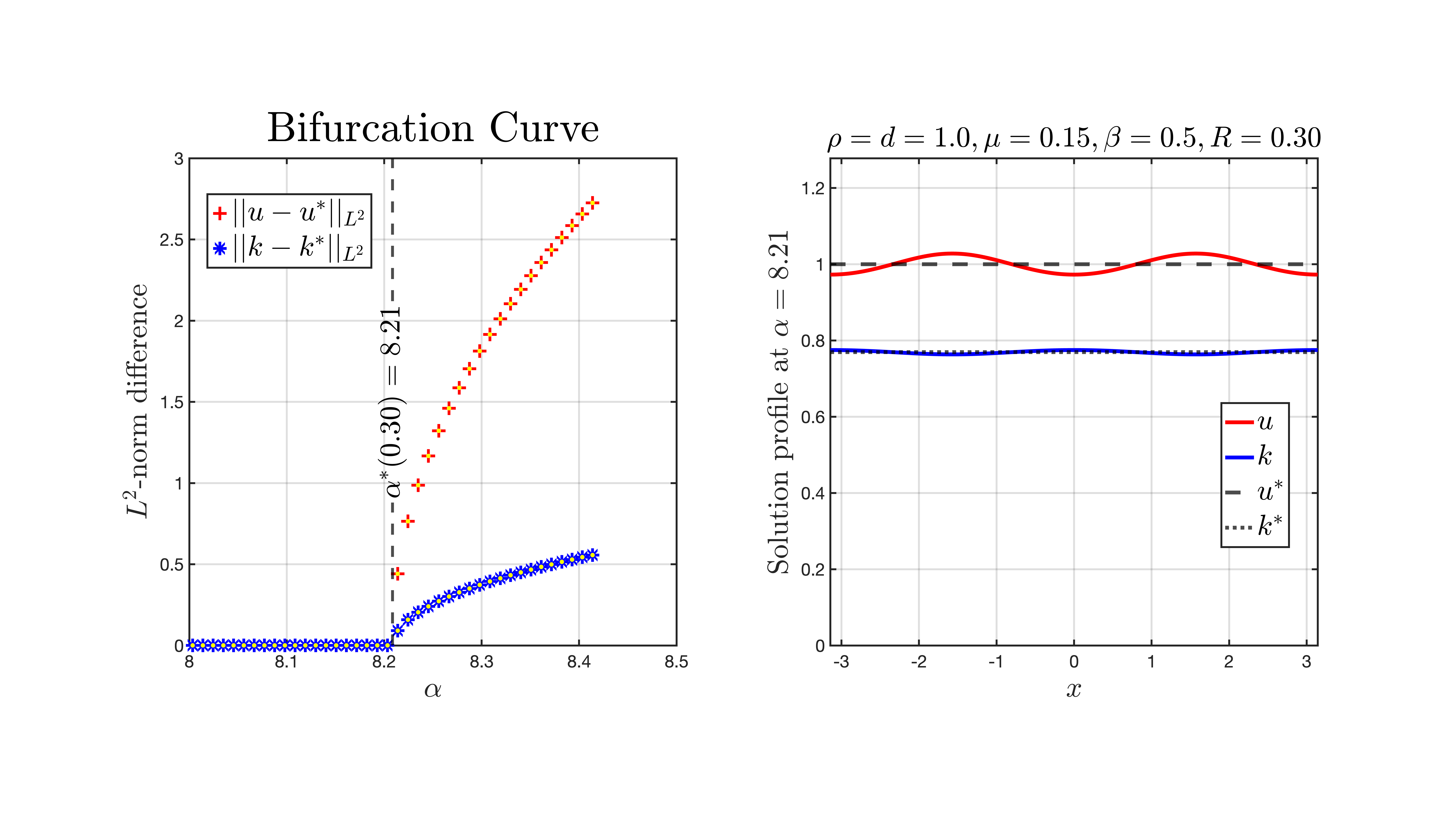}
    \caption{A bifurcation diagram near the critical threshold $\alpha^*$ when $R=0.30$ for \textbf{Example} \ref{ex2} of Section \ref{sec:apps} (left panel) and the solution profile at steady state just beyond the critical threshold (right panel).}
    \label{fig:6}
\end{figure}

\begin{figure}
    \centering
    \includegraphics[trim={4cm 4cm 5cm 4cm},clip,width=\linewidth]{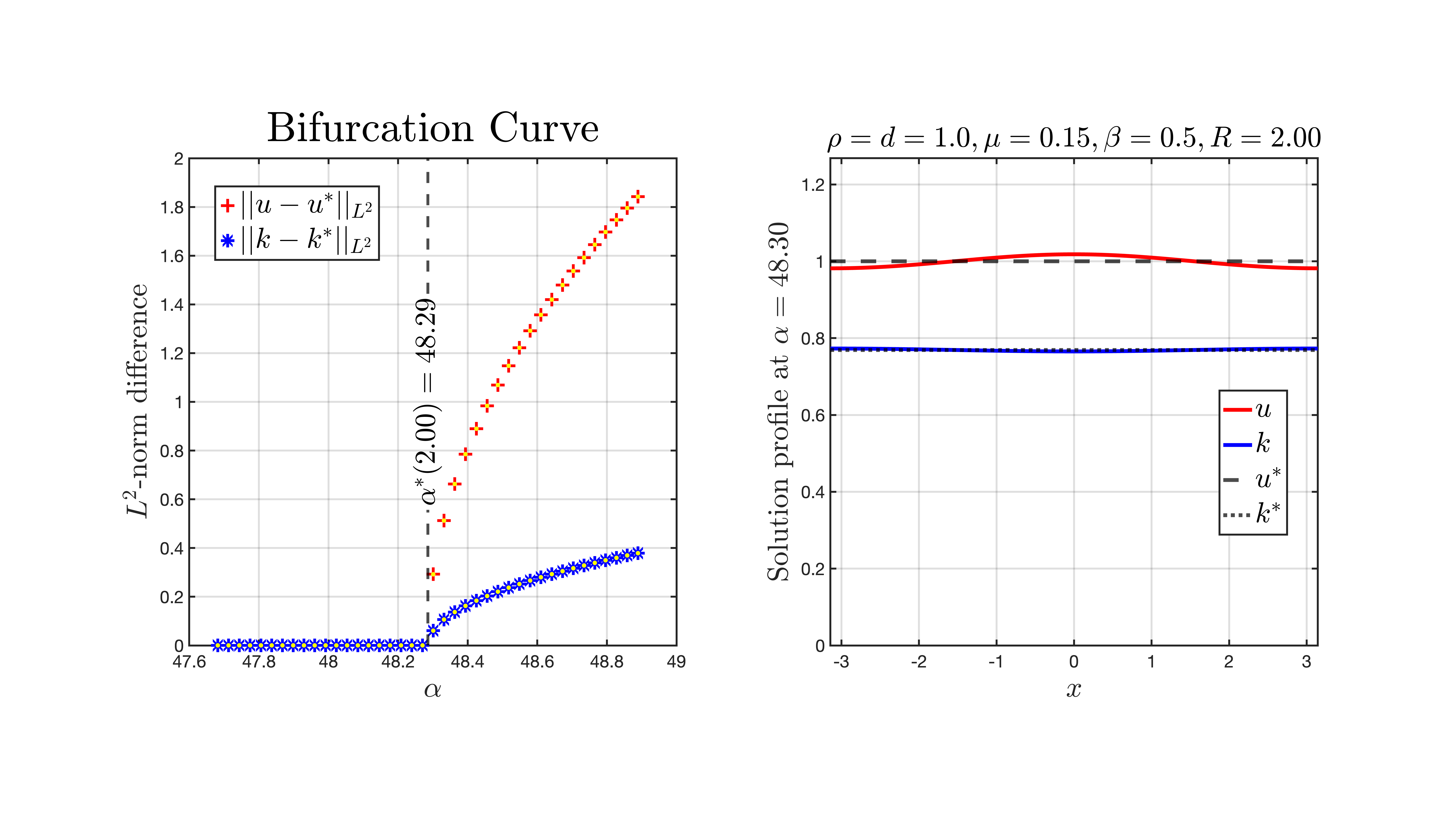}
    \caption{A bifurcation diagram near the critical threshold $\alpha^*$ when $R=2.00$ for \textbf{Example} \ref{ex2} of Section \ref{sec:apps} (left panel) and the solution profile at steady state just beyond the critical threshold (right panel).}
    \label{fig:7}
\end{figure}
 
\begin{examp}\label{ex3}
We conclude with an illuminating, though somewhat pathological, example to accentuate the solution behavior observed in \textbf{Example} \ref{ex1} and \textbf{Example} \ref{ex2}. For fixed parameters $0 < \varepsilon \ll 1 < \gamma < \infty$, we first define the auxiliary function $\tilde g : \mathbb{R}^+ \rightarrow (0,\varepsilon)$ by
$$
\tilde g(u) := \frac{\varepsilon}{2} \left( 1 + \frac{2}{\pi} \arctan [\gamma (u - u_*)] \right).
$$
This is a standard approximation of a step function, transitioning from $0$ to $\varepsilon$ near $u_*$, where the steepness of the transition is controlled by $\gamma$. Moreover, we have $\tilde g(u_*) = \varepsilon/2$ and $\tilde g^\prime (u_*) = \gamma \varepsilon/\pi$. 

We then fix $g_1(u) = g(u)$ for some smooth $g :\mathbb{R}^+ \rightarrow \mathbb{R}^+$ satisfying $g(0) = 0$ and $g^\prime (u_*) / g(u_*) > 0$. We then let $g_2(u) = \mu + \beta u + \tilde g(u)$ so that $\tilde g$ acts as a small perturbation of $\mu + \beta u$. One can verify that this combination always satisfies Hypotheses \textbf{\ref{H2a}}-\textbf{\ref{H2b}}. Then, the sign of $w_{u*}$ depends critically on
$$
\frac{g^\prime (u_*)}{g(u_*)} - \frac{\beta + (\gamma \varepsilon)/\pi}{\mu + \beta u_* + \varepsilon/2}.
$$
Therefore, for $\gamma$ sufficiently small the sign of $w_{u_*}$ is the same as the unperturbed case, i.e., $w_{u*} > 0$ as in \textbf{Example} \ref{ex1}. Alternatively, for $\varepsilon$ fixed we may choose $\gamma$ sufficiently large so that $w_{u*}<0$ as in \textbf{Example} \ref{ex2}. \textit{This means that the function $g_2(u) = \mu + \beta u + \tilde g (u)$, a small perturbation that can be made arbitrarily close to $\mu + \beta u$, can have entirely different dynamics than the unperturbed case.} For modelling movement and spatial memory, this highlights an interesting property of the population dynamics being described: if the excitatory and adaptory encoding processes $g_1$ and $g_2$, respectively, have a sharp enough interface at the state $u_*$, this will determine the direction for which aggregation will produce a patterned state. This indicates an intimate relation between the memory encoding processes and the description of the population-level dynamics, such as a change in habitat quality (e.g., a decreased carrying capacity) resulting in a change in the utility of the spatial map $k$. This also exemplifies that the \textit{relative} rates of change of $g_i$ in the map $k$ is the determining factor, not the \textit{absolute} values in the map $k$. We hope to explore the consequences of such observed behaviour in future work.
\end{examp}

\section{Without population growth dynamics}\label{sec:nogrowth}

\noindent To conclude, we briefly discuss system \eqref{1} without the reaction term, taking the form
\begin{equation}\label{1'}
	\begin{cases}
	u_t = d u_{xx} + \alpha ( u \overline{k}_{x} )_{x}, & x\in \mathbb{R}, \ t>0, \\
	k_t = g_1(u) - g_2(u) k, & x \in \mathbb{R}, \ t> 0, \\
	u(x,0)=u_0(x), \; k(x,0)=k_0(x),& x \in \mathbb{R}. \\
	\end{cases}
	\end{equation}
According to our prior analysis, the even, $2\pi$-periodic classical solution of system \eqref{1'} is equivalent to the following local parabolic-ordinary-elliptic system
\begin{equation}\label{1.18'}
   \begin{cases}
U_t = d U_{xx} + \alpha ( U V_{x})_{x}, & x\in (0,\pi), \ t>0, \\
K_t = g_1(U) - g_2(U) K, & x \in (0,\pi), \ t> 0, \\
0 = V_{xx}-\frac{1}{R^2}(V-K), & x \in (0,\pi), \ t> 0,\\
U_x(0,t)=U_x(\pi,t)=0,\ V_x(0,t)=V_x(\pi,t)=0,\ & t>0,
\end{cases} 
\end{equation} 
with initial data $(U(x,0), K(x,0), V(x,0)) =(u_0,k_0,G*k_0)$ and $V:=G*K$.

In this case, we only need Hypothesis \textbf{\ref{H2a}} and \textbf{\ref{H2b}}. The constant steady state of \eqref{1.18'} is
$$u_*=\frac{1}{\pi}\int_0^\pi u_0(x){\rm d}x, \ k_*=\frac{g_1(u_*)}{g_2(u_*)},\ v_*=k_*.$$ 
Notice that compared with the case when growth dynamics are included, the constant state is now determined by the total initial mass rather than the zero of the growth response $f(\cdot)$. From a linear stability analysis, we can define the critical aggregation strength for each wavenumber $n$ as in \eqref{1.13}:
\begin{equation}\label{1.13'}
    \alpha_n(R) := \dfrac{(1+n^2 R^2)d w_{k*}}{u_* w_{u*}}.
\end{equation}
If $w_{u*}>0\, (<0)$, then $\alpha_n(R)<0\, (>0)$ and $\alpha_n(R)$ is monotonically decreasing (increasing) with respect to $n$ for fixed $R$. We can then define critical aggregation strengths exactly as in \eqref{alpha*}; the key difference is that in the absence of population growth dynamics, a factor of $n^2$ vanishes and the critical wavenumber at which the maximum/minimum occurs is $n=1$, i.e., $\alpha_* (R) = \alpha_1 (R)$ (or $\alpha^*(R) = \alpha_1(R)$). Therefore, the patterned state, whenever it occurs, will always have a single peak, where the sign of $\alpha_1(R)$ (and the phase relationship between $u$ and $k$) is determined by the sign of $w_{u*}$.

Therefore, the local stability theorem of the constant steady state of system \eqref{1.18'} is the same as Theorem \ref{thm:1.2}. The bifurcation analysis now deviates from the previous one. 
Here we define a nonlinear mapping $F': \mathbb{R} \times X \times Y\times X \rightarrow Y^4$ by
\begin{equation}
F'(\alpha,u,k,v)=\left(\begin{array}{c}
d u_{xx} + \alpha ( u v_{x})_{x}
\\ 
g_1(u) - g_2(u) k\\
v_{xx}-\frac{1}{R^2}(v-k)\\
\int_0^\pi u(x) {\rm d}x-\int_0^\pi u_0(x) {\rm d}x
\end{array}\right).
\end{equation}
The Fr\'{e}chet derivative of $F'$ with respect to $(u,k,v)$ at $(u_*,k_*,v_*)$ is
\begin{equation}
    F'_{(u,k,v)}(\alpha_n,u_*,k_*,v_*)\left(\begin{array}{c}
\phi
\\ 
\varphi\\
\psi
\end{array}\right)=
\left(\begin{array}{c}
d \phi_{xx} + \alpha_n u_* \psi_{xx} 
\\ 
w_{u*} \phi + w_{k*} \varphi
\\
\psi_{xx}-\frac{1}{R^2}(\psi-\varphi)\\
\int_0^\pi \phi(x) {\rm d}x
\end{array}\right).
\end{equation}
Noting that $\mathcal{N}(F'_{(u,k,v)}(\alpha_n,u_*,k_*,v_*))=\mathcal{N}(F_{(u,k,v)}(\alpha_n,u_*,k_*,v_*))$. 
Therefore, the subsequent bifurcation analysis can be carried out similarly to the above analysis (in this case $f(u)\equiv0$), and Theorems \ref{thm:bif2} and \ref{thm:bif3} can be obtained similarly.

\section*{Acknowledgments}
YS is supported by the Natural Sciences and Engineering Research Council of Canada (NSERC Grant PDF-578181-2023). DL is supported by the China Postdoctoral Science Foundation under Grant Number 2024M763690. JS is supported by US-NSF grant OCE-2207343. HW gratefully acknowledges support from the Natural Sciences and Engineering Research Council of Canada (Discovery Grant RGPIN-2020-03911 and NSERC Accelerator Grant RGPAS-2020-00090) and the Canada Research Chairs program (Tier 1 Canada Research Chair Award).

 \begin{appendices}
 \section{Appendix}\label{sec:appendix}

    \subsection{Equivalence of Systems }\label{subsec:A1}

\begin{proof}[Proof of Proposition \ref{prop:equivalentsystem}] We start with statement i.). To this end, suppose $(u,k)$ is an even, $2\pi$-periodic classical solution solving system \eqref{1} with initial data $(u_0,k_0)$ even and satisfying \eqref{Hinitialdata}. Take $U,K : (0,\pi) \mapsto \mathbb{R}^+$ defined by
$$
U (x,t) = u(x,t) \bigr\vert_{(x,t) \in (0,\pi) \times (0,T)}, \quad K (x,t) = k(x,t) \bigr\vert_{(x,t) \in (0,\pi) \times (0,T)}.
$$

By the $2\pi$-periodic boundary and evenness of the solution $(u,k)$, there holds $u_x (0,t) = u_x(\pi,t) = 0$.  

Since $K$ is smooth, we can then define the auxiliary function $V := G * K$, which inherits the smoothness of $K$. Direct computation yields $V$ satisfying the following system
\begin{equation}\label{V}
   \begin{cases}
V_{xx}=\frac{1}{R^2}(V-K), & x \in (0,\pi), \ t> 0,\\
V_x(0,t)=V_x(\pi,t)=0,\ & t>0.
\end{cases} 
\end{equation} 
Therefore, we observe that the trio $(U,K,V)$ is a solution to the local Neumann problem \eqref{1.18} subject to the initial data $(U(x,0), K(x,0), V(x,0)) = (u_0(x), k_0(x), G* k_0(x))$. This proves part i.).

To prove ii.), we now assume that $(U,K,V)$ is a classical solution to the parabolic-ordinary-elliptic  problem \eqref{1.18} in $(0,\pi)$ with initial data $(u_0(x), k_0(x), G* k_0(x))$ defined in $(0,\pi)$ satisfying a homogeneous Neumann boundary condition. We then reflect $(U,K,V)$ across $0$ to define the auxiliary functions $(u,k,v)$ such that $$(u(-x,t), k(-x,t), v(-x,t)) = (U(x,t), K(x,t), V(x,t))$$ for all $t>0$, for each $x \in (0,\pi)$. Then we extend $(u,k,v)$ periodically over the whole space so that $(u,k,v) : \left[\mathbb{R} \times (0,T) \right]^3 \mapsto \left[ \mathbb{R}^+ \right]^3$ is well-defined. $(u,k,v)$ is then $2\pi$-periodic defined over all of $\mathbb{R}$, and satisfies the equations of \eqref{1.18} pointwise.

From the equation satisfied by $v$, we may solve for the general solution, which must be of the form
$$
v(x,t) = \frac{1}{2R} \int_\mathbb{R} e^{|x-y|/R} k(y,t) {\rm d}y + A(t) e^{x / R} + B(t) e^{-x/R} ,
$$
for some coefficients $A(t)$, $B(t)$. Since $v$ is spatially periodic, we find that $A \equiv B \equiv 0$ for all $t>0$, and so the unique solution $v$ is given precisely by
$$
v(x,t) = G * k(x,t) .
$$
Therefore, the pair $(u,k)$ is even, smooth, $2\pi$-periodic, satisfies $(u(x,0) , k(x,0)) = (u_0, k_0)$, and satisfies the equations of system \eqref{1}. Therefore, $(u,k)$ is an classical, even, $2\pi$-periodic solution of system \eqref{1}.
\end{proof}

\subsection{Proof of Lemma \ref{lem:uniqueness}}\label{sec:AppUniqueness}

\begin{proof}[Proof of Lemma \ref{lem:uniqueness}]
    We provide the details for the equivalent system \eqref{1.18}; an identical approach applies to the original system \eqref{1}. Suppose there are two strong solutions with common initial data, $(u_1, k_1, v_1)$ and $(u_2, v_2, k_2)$, solving problem \eqref{1.18}. We first set $U = u_1 - u_2$, $K = k_1 - k_2$, $V = v_1 - v_2$. By linearity $V$ satisfies
    $$
- V_{xx} + \frac{1}{R^2} V = \frac{1}{R^2} K .
    $$
    Multiplying by $V$, integrating by parts, and using Cauchy's inequality,y we have that
\begin{align}\label{est:u1.1}
    \int_\Omega \as{V_x}^2 \dx \leq \frac{1}{2R^2} \int_\Omega K^2 \dx.
\end{align}
    
    Taking the derivative of $\tfrac{1}{2} \left( \norm{U(\cdot,t)}_{L^2(\Omega)} + \norm{K(\cdot,t)}_{L^2(\Omega)} \right)$ with respect to time and integrating by parts yields
    \begin{align}\label{est:u1.3}
        \frac{1}{2} \frac{{\rm d}}{{\rm d}t} \int_\Omega (U^2 + K^2 ) \dx = & - d \int_\Omega \as{U_x}^2 \dx + \alpha \int_\Omega U_x \left[ u_1 (v_1)_x - u_2 (v_2)_x \right] \dx \nonumber \\
        + & \, \int_\Omega U^2 \left( \frac{f(u_1) - f(u_2)}{u_1 - u_2} \right) \dx \nonumber \\
        + & \int_\Omega K \left[ \left( \frac{g(u_1) - g(u_2)}{u_1 - u_2} \right) U - \beta (u_1 k_1 - u_2 k_2) \right] \dx \nonumber \\
        -& \, \mu \int_\Omega K^2 \dx .
    \end{align}
    We control each term on the right hand side via Cauchy's inequality. First, we can control the higher order terms using the boundedness of the strong solution over $Q_T$:
    \begin{align*}
        \alpha U_x \left[ u_1 (v_1)_x - u_2 (v_2)_x \right] =&\, \alpha \left[ U_{x} U (v_1)_x + U_x u_2 V_x  \right] \nonumber \\
        \leq &\, d \as{U_x}^2 + C_1 ( U^2 + \as{V_x}^2 ) ,
    \end{align*}
    where $C_1$ depends on $\alpha$, $d$, and the uniform bounds of $(v_1)_x$, $u_2$ over $Q_T$. Paired with estimate \eqref{est:u1.1}, there holds
    \begin{align}
        \alpha \int_\Omega U_x \left[ u_1 (v_1)_x - u_2 (v_2)_x \right] \dx &\leq d \int_\Omega \as{U_x}^2 \dx + C_1 \int_\Omega U^2 \dx + \frac{C_1}{R^2} \int_\Omega K^2 \dx .
    \end{align}
    We then estimate the lower order terms. Since $u_1$, $u_2$ are bounded over $Q_T$, the smoothness of $f,g$ (local Lipschitz continuity is sufficient) gives that
    $$
\sup_{u_1, u_2 \geq 0} \frac{f(u_1) - f(u_2)}{u_1 - u_2},\quad \sup_{u_1, u_2 \geq 0}\frac{g(u_1) - g(u_2)}{u_1 - u_2} \leq C_2,
    $$
    where $C_2$ is bounded for any fixed $T>0$. Lastly,
    $$
\beta K ( u_1 k_1 - u_2 k_2 ) = \beta K U k_1 + K^2 u_2 \leq C_3 ( K^2 + U^2),
    $$
 where $C_3$ depends on $\beta$ and the uniform bounds on $k_1$, $u_2$ over $Q_T$. Together, estimate \eqref{est:u1.3} becomes
 \begin{align}
            \frac{1}{2} \frac{{\rm d}}{{\rm d}t} \int_\Omega (U^2 + K^2 ) \dxi \leq & C_4 \int_{\Omega} ( U^2 + K^2 ) \dxi .
 \end{align}
 Gr{\"o}nwall's lemma implies that $U \equiv K \equiv 0$. Estimate \eqref{est:u1.1} then implies that $V$ is constant over $Q_T$ and is, therefore, identically zero, and uniqueness is proven.
\end{proof}

\subsection{Expression of \texorpdfstring{$\alpha_n'(0)$}{} and \texorpdfstring{$\alpha_n''(0)$}{}}\label{sec:AppdExpression}
 
The bifurcation direction of the bifurcating solutions in $\Gamma_n$ can be determined by the formula in \cite{Shi1999}, 
\begin{equation}
\begin{aligned}
\alpha_n'(0) &= -\dfrac{\langle l,F_{(u,k,v)(u,k,v)}(\alpha_n,u_*,k_*,v_*))[q_n, q_n]\rangle}{2\langle l,F_{\alpha(u,k,v)}(\alpha_n,u_*,k_*,v_*))[q_n]\rangle} 
\end{aligned}
\end{equation}
where $l\in Y^3$ is a linear function satisfying 
\begin{equation}
    \langle l,(h_1,h_2,h_3) \rangle = \int_0^{\pi} (h_1 + M^*_1 h_2 + M^*_2 h_3) \cos(nx) \dx .
\end{equation}
From 
\begin{equation}\label{1.36}
\begin{aligned}
  &\langle l,F_{(u,k,v)(u,k,v)}(\alpha_n,u_*,k_*,v_*))[q_n, q_n]\rangle\\
  = & \int_0^\pi [2 \alpha_n n^2 M_2 (\sin^2(nx)-\cos^2(nx))+f_{uu*}\cos^2(nx)]  \cos(nx) \dx\\
  &+\int_0^\pi M^*_1 (w_{uu*} + 2 w_{uk*} M_1 + w_{kk*} M_1^2 )\cos^2(nx) \cos(nx) \dx \\
  = & 0,
\end{aligned}
\end{equation}
we obtain $\alpha_n'(0)=0$. Therefore, a pitchfork bifurcation occurs at $\alpha=\alpha_n$.
We need to calculate $\alpha_n''(0)$ to determine the direction of the bifurcation, if $\alpha_n''(0)>0\;  (\alpha_n''(0)<0)$, the bifurcation is forward (backward). Moreover, $\alpha_n''(0)$ can be calculated as follows,
\begin{equation}\label{3.10}
\begin{aligned}
\alpha_n''(0) = & -\dfrac{\langle l,F_{(u,k,v)(u,k,v)(u,k,v)}(\alpha_n,u_*,k_*,v_*))[q_n, q_n, q_n]\rangle}{3\langle l,F_{\alpha(u,k,v)}(\alpha_n,u_*,k_*,v_*))[q_n]\rangle}\\
& -\dfrac{\langle l,F_{(u,k,v)(u,k,v)}(\alpha_n,u_*,k_*,v_*))[q_n,\Theta]\rangle}{\langle l,F_{\alpha(u,k,v)}(\alpha_n,u_*,k_*,v_*))[q_n]\rangle}
\end{aligned}
\end{equation}
where $\Theta=(\Theta_1,\Theta_2,\Theta_3)$ is the unique solution of 
\begin{equation}\label{3.11}
    F_{(u,k,v)(u,k,v)}(\alpha_n,u_*,k_*,v_*))[q_n, q_n]+F_{(u,k,v)}(\alpha_n,u_*,k_*,v_*))[\Theta]=0.
\end{equation}
We assume $\Theta=(\Theta_1,\Theta_2,\Theta_3)$ has the following form by \cite{ShiShiWang2021JMB} and \eqref{1.36}, 
\begin{equation}\label{3.12}
    \Theta_i=\Theta_i^1+\Theta_i^2 \cos(2nx),\ \ i=1,2,3. 
\end{equation}
Substituting \eqref{3.12} into \eqref{3.11}, we have
\begin{equation*}
    \begin{aligned}
        \left(\begin{array}{c}
f_{uu*}\cos^2(nx)-2\alpha_n M_2 n^2 \cos (2nx)
\\ 
(w_{uu*}+2w_{uk*}M_1+w_{kk*}M_1^2)\cos^2(nx)
\\
0
\end{array}\right)
+
\left(\begin{array}{c}
d(\Theta_1)_{xx}+f_{u*}\Theta_1+\alpha_n u_* (\Theta_3)_{xx}
\\ 
w_{u*}\Theta_1+w_{k*}\Theta_2
\\
(\Theta_3)_{xx}-\frac{1}{R^2}(\Theta_3-\Theta_2)
\end{array}\right)=0.
    \end{aligned}
\end{equation*}
From the above equation, we see that 
\begin{equation}\label{theta}
\begin{aligned}
    \Theta_1^1=&\Theta_1^2=-\frac{f_{uu*}}{2f_{u*}},\ \Theta_3^2=\frac{ M_2 \alpha_n n^2 f_{u*}-d n^2 }{2\alpha_n u_* n^2f_{u*}},\\
    \Theta_2^1&=\Theta_2^2=\Theta_3^1=(1+4n^2R^2)\Theta_3^2.
    \end{aligned}
\end{equation}
After tedious calculations, 
\begin{equation}\label{alpha'}
\begin{aligned}
       \alpha_n''(0)=& \frac{\dfrac{f_{uuu*}+M_1^*\overline{w}}{4}+f_{uu*}\left(\Theta_1^1+\dfrac{\Theta_1^2}{2}\right)-M_2 \alpha_n n^2\left(\Theta_1^1-\dfrac{\Theta_1^2}{2}\right)}{M_2 u_* n^2}\\
        & +\frac{-\alpha_n n^2 \Theta_3^2+M_1^*\left[w_1\left(\Theta_1^1+\dfrac{\Theta_1^2}{2}\right)+w_2 \left(\Theta_2^1+\dfrac{\Theta_2^2}{2}\right)\right]}{M_2 u_* n^2},
\end{aligned}
\end{equation}
where 
\begin{align*}
   f_{uu*}=f_{uu}(u_*),\ \ f_{uuu*}:=f_{uuu}(u_*),\ \ w_1=w_{uu*}+M_1w_{ku*},\ \ w_2=w_{uk*}+M_1w_{kk*},
\end{align*}
\begin{align*}
    \overline{w}=&[(w_{uuu*}+M_1(w_{kuu*}+w_{uku*})+M_1^2 w_{kku*})\\
    &+(w_{uuk*}+M_1(w_{kuk*}+w_{ukk*})+M_1^2 w_{kkk*})],
\end{align*}
and
\begin{align*}
  &  w_{uu*}:=w_{uu},\ w_{uk*}:=w_{uk},\  
  w_{ku*}:=w_{ku},\
  w_{kk*}:=w_{kk},\\ 
  &  w_{uuu*}:=w_{uuu},\
     w_{uuk*}:=w_{uuk},\
     w_{uku*}:=w_{uku},\ w_{ukk*}:=w_{ukk},\\ & w_{kuu*}:=w_{kuu},\
     w_{kuk*}:=w_{kuk},\ w_{kku*}:=w_{kku},\ w_{kkk*}:=w_{kkk},
\end{align*}
and $M_1, M_2$, and $M_1^*$ are defined in \eqref{1.24} and \eqref{M1*M2*}.
 \end{appendices}
	
	
	\bibliographystyle{elsarticle-num-names} 
	\bibliography{ref}
	
	
	
	
\end{document}